\documentclass[preprint,12pt]{elsarticle}
\usepackage{amsmath}
\usepackage{amsfonts}

\usepackage{amssymb}

\newtheorem{theorem}{Theorem}
 \newtheorem{lemma}{Lemma}
 \newdefinition{remark}{Remark}
 \newdefinition{definition}{Definition}
\newproof{proof}{Proof}
\newtheorem{example}{Example}%
\newtheorem{Assumption}{Assumption}%

\journal{Stochastic Processes and their Applications}

\begin{document}

\begin{frontmatter}



\title{Pathwise convergence of  a novel  numerical scheme  based  on semi-implicit method for stochastic differential-algebraic equations with  non-global Lipschitz coefficients}


\author[gt]{Guy Tsafack} 
\ead{guy.tsafack@hvl.no (guytsafack4@gmail.com)}
\address[gt]{Department of Computing Mathematics and Physics,  Western Norway University of Applied Sciences, Inndalsveien 28, 5063 Bergen, Norway.}
\author[gt,at1]{Antoine Tambue} 
\ead{antoine.tambue@hvl.no (antonio@aims.ac.za)}
\address[at1]{The Department of Mathematics `\& Applied Mathematics, University of Cape Town,Private Bag, 7701
Rondebosch, Cape Town, South Africa }

\begin{abstract}
This paper delves into the well-posedness and the numerical approximation of non-autonomous stochastic differential algebraic equations (SDAEs) with nonlinear local Lipschitz coefficients  that satisfy the more general monotonicity condition called Khasminskii condition. The key challenge is the presence of a singular matrix which makes the numerical integration hard and heavy. To address this challenge, we propose a novel numerical scheme based on semi-implicit method for the drift component of the SDAEs.  More precisely we split the drift term as the sum of a linear term and a nonlinear term. The linear part is approximated implicitly, while the nonlinear part is approximated explicitly. The linear component's role is to handle the singularity issues during the numerical integration without  the resolution of nonlinear algebraic equations in the constraint equations.  This novel scheme is therefore very efficient   for  SDAEs in high dimension that come after the spatial discretisation of  stochastic  partial differential algebraic equations (SPDAEs).  To prove the pathwise convergence  of our novel scheme, we first  derive a equivalent scheme called dual scheme,  suitable for mathematical analysis and  linked to the inherent stochastic differential equation resulting from the elimination of constraints in the initial SDAEs.  We  prove  that  our novel  scheme  converges  to the exact solution with rate  $\frac{1}{2}-\epsilon$, for arbitrary $\epsilon>0$ in the pathwise sense.  Numerical simulations are performed to demonstrate  the efficiency of the scheme in high dimension and to  show that our theoretical results are in agreement with numerical experiments.
\end{abstract}



\begin{keyword}


 Non-autonomous stochastic differential-algebraic equations\sep nonlinear and locally Lipschitz coefficients\sep index-1 differential-algebraic equations \sep Pathwise convergence.
\end{keyword}

\end{frontmatter}



\section{Introduction}\label{sec1}
We are investigating the numerical approximation of the following stochastic differential algebraic equation
\begin{equation}\label{equa1a}
	A(t)dY(t)=\left[B(t)Y(t)+f(t,Y(t))\right]dt + g(t,Y(t))dW(t), \quad Y(0)=\zeta, \quad  t\in \left[0,T\right].
\end{equation}

Here $T>0,\quad T\ne \infty$,  A, B : $ \left[0,T\right] \to \mathcal{M}_{d\times d}(\mathbb{R})$ are continuously differentiable and $ A(t)$ is a singular matrix for all $t\in \left[0,T\right]  $ and $\zeta\in \mathbb{R}^d$.  
The functions  $f:\left[0,T\right]\times D \to \mathbb{R}^d  $ (drift) and $g:\left[0,T\right]\times D\to \mathbb{R}^{d\times d_1} $ (diffusion)  are Borel functions on  $\left[0,T\right]\times D$ with  $D\subseteq\mathbb{R}^d$. 
The process $W(\cdot)$ is a $d_1$-dimensional  Wiener process defined on the probability space $(\Omega, \mathcal{F},\mathbb{P})$ with the natural filtration  $\left( \mathcal{F}_t\right)_{t\geq 0} $. 
The unknown function  $Y(\cdot)$ is a vector-valued stochastic process of dimension  $d$ and depends both on time $ t\in \left[0,T \right] $ and the sample $\omega \in \Omega$. To ease the notation, the argument $\omega$ is omitted. 
Our assumptions  on \eqref{equa1a} will  allow  the solution $Y(t)$ to never leave $D$, therefore the values of $\mu=B +f$ and $g$ outside $D$ are irrelevant and just for convenience we define $\mu(t,Y(t))=B(t)Y(t)+f(t,Y(t))=0_{d}$\footnote{null vector in $\mathbb{R}^d$ } and $ g(t,Y(t))=0_{d\times d_1}$\footnote{ null in matrix  in $\mathbb{R}^{d \times d_1}$ } for $ Y \notin   D, \;\; t \in \left[0,T\right]$.

Note that if the matrix $ A(t)$ is  inversible, we have the standard  Stochastic differential equations (SDEs). For  SDEs,  It is well known that   for smooth  and bounded coefficients  in $\mathbb{R}^d$, the convergence in the $p$th mean  norm  with $p \geq 1$ \footnote{the strong convergence} with order $\alpha$  implies the pathwise convergence  with order $\alpha-\epsilon$ for arbitrary $\epsilon>0$  \cite{Koeden}. Indeed  the  theory of random dynamical systems is of pathwise nature  and the $p$th mean  norm  of the solution $p \geq 1$ do not always exist \cite{Jentzen}, so the strong and weak convergences  cannot be studied  in the $p$th mean  norm  setting. 
  Pathwise convergence is an alternative to the strong and weak convergences  in  the literature. It arises naturally in many important applications
    \cite {Jentzen,Koeden},  in addition  numerical calculations are  always carried out pathwise. 
The resolution of the simple differential algebraic equations (DAEs) is not easy due to the singularity of the matrix $A(\cdot)$ \cite{gear1988differential,rabier2002theoretical,kunkel2006differential}. The problems become more challenging when we add  noise to the  DAEs. This combination is well known as stochastic differential algebraic equations (SDAEs) \cite{denk2008efficient, riaza2008differential}.  Many researchers have been interested in analyzing this type of equation with index-1 and constant $A(t)=A$. For example, under the global Lipschitz conditions, the strong convergence for numerical methods applied to stochastic differential algebraic equations (SDAEs) with index-1 was established in \cite{Renate2003, kupper2012runge,kupper2015stability}.
The basic idea used in those papers was to transform the SDAEs into the SDEs and some constraints. Using some well-known methods in the case of SDEs, they could generalize in the case of SDAEs to obtain the convergence results \footnote{in the means-square sense}.  Clearly, in the high dimension, this transformation can  time consuming due to the number of equations. This is for example for  SDAEs that come after the spatial discretisation of  stochastic  partial differential algebraic equations (SPDAEs). Moreover, many important SDAEs models satisfy only the local Lipschitz conditions \cite{serea2025existence,Mao2007}. As we have mentioned for SDEs,  the $p$th mean  norm  of the solution $p \geq 1$ do not always exist, so pathwise convergence will be the  only alternative. To the  best of  our knowledge, pathwise convergence for  general SDAEs  given  in \eqref{equa1a} have been lacked in the literature.
 
%

%
The existence and uniqueness of the solution for SDAEs   \eqref{equa1a}  under the local Lipschitz condition and monotone conditions have been investigated in \cite{serea2025existence}. Unfortunately, the classical  monotonicity condition is  just a special case of the Khasminskii condition \cite{Khasminskii, luo2011generalised}. Indeed the Khasminskii condition is a more general requirement, especially when dealing with local Lipschitz conditions to ensure the existence of a continuous-time solution or process $Y(\cdot)$ \cite{luo2011generalised}. 

 This  paper has two contributions, the well-posedness of SDAEs under Khasminskii condition and the rigorous mathematical study of pathwise  convergence in the context of SDAEs with local Lipschitz conditions.   Indeed, the presence of singular matrices in these equations poses considerable challenges for both theoretical analysis and simulation methods.  Firstly, we establish the existence and uniqueness of the solution for SDAEs under the local Lipschitz and the Khasminskii conditions, and secondly, we develop  a semi-implicit Euler scheme for numerical approximation of  SDAEs and provide the pathwise  convergence analysis of the scheme.   It is important to mention  that the scheme can fail to convergence strongly \cite{HUT} because of the non global Lipschitz conditions of the coefficients, so the pathwise  convergence is the only alternative. 
 
 For the well-posedness of  SDAEs  \eqref{equa1a}, we demonstrate that the coefficients of the corresponding SDEs obtained after transforming the initial SDAEs also satisfy the Khasminskii conditions, and use   well known results for  SDEs  in  \cite{Khasminskii, luo2011generalised}.
For numerical integration, we build a numerical scheme based on  semi implicit method  and  transform that  semi-implicit  scheme for SDAEs  into a equivalent scheme called  dual semi-implicit  for  the corresponding  inherent  SDEs by eliminating the constraints. The equivalent  dual scheme is then used  for the mathematical analysis and the initial semi-implicit scheme  for SDAEs is used  numerical simulations. The advantage of this method is that  to simulate,  there is no need to solve systems of nonlinear algebraic equations at every time step as  \cite{Renate2003, kupper2015stability, kupper2012runge}, which is very costly, mostly in high dimension. 

This paper is organized as follows. In the second part, we present some definitions and the theorem of the existence and uniqueness of solutions for SDAEs under local Lipschitz and Khasminskii conditions. 
The third part of this paper focuses on the numerical method.
We develop  our  semi-implicit Euler scheme for the numerical approximation of  SDAEs and provide a pathwise convergence analysis.  Finally, we present some simulation results in order to motivate  and  check our theoretical results.

\section{Definitions and well-posedness of the SDAEs}
The goal here is to provide keys definitions and notations that will be used in this work along with the well-posedness result of    SDAEs \eqref{equa1a}  under Khasminskii condition. \\
Here,  we will consider $D=\mathbb{R}^d$ and  the drift part will be denoted by $$\mu(t,Y(t))=B(t)Y(t)+f(t,Y(t)).$$
The well-posedness for general domain  $D\subseteq\mathbb{R}^d$ can easily be updated by following \cite[Theorem 3.5]{Khasminskii}.

\subsection{Definitions and  Notations}
\begin{definition}\label{def2b}
	A strong solution of \eqref{equa1a}  is a process $Y(\cdot) = (Y(t))_{t\in \left[0,T \right] }$ with continuous sample paths that respects the following conditions:\\
	\begin{itemize}
		\item [(i)]  $Y(\cdot)$ is adapted to the filtration $\left\lbrace \mathcal{F}_t\right\rbrace_{t\in \left[0,T\right] } $,\\
		\item [(ii)] $\int_{0}^t|\mu_i(s,Y(s))|ds< \infty$ a.s, for all   $ i=1,..., d$,  $t\in \left[0,T\right]$,\\
		\item[ (iii)] $\int_{0}^tg_{ij}^2(s,Y(s))dW_j(s)< \infty$ a.s, for all  $i=1,...,d; j=1,...,d_1$, $t\in \left[0,T\right]$,\\
		\item[(iv)]  $Y(\cdot) = (Y(t))_{t\in \left[0,T \right] }$ satisfies the following equation
		\begin{eqnarray}
			\label{reps}
			A(t)Y(t)=A(0)Y_0+\int_0^tA'(s)Y(s)+\mu(s, Y(s))ds+\int_0^tg(s, Y(s))dW(s).
		\end{eqnarray} 
	\end{itemize}		
\end{definition} 
Note that for having the representation \eqref{reps}, we have  applied  the It\^o lemma  to the  $k(t,Y)=A(t)Y$. This means that $dk(t,Y)=A'(t)Ydt+A(t)dY$.
\newline\\
\begin{definition}\label{def2b}
	The SDAEs (\ref{equa1a}) is said to be of index-1  if the noise sources do not appear in the constraints, and the constraints (AEs) are globally uniquely solvable. 
\end{definition}  A suitable hypothesis for having this assumption is, for example, that 
$$ \text{Im} g(t, Y)\subseteq \text{Im}A(t),\;\;\;  \text{for all }\;\;\; Y\in \mathbb{R}^d\;\;\; ,  t\in \left[0, T\right]. $$
We denote by $$J(t,Y)=A(t)+R(t)\mu_{Y}(t,Y),\;\;  t\in \left[0,T\right],\;\;\;\; Y\in \mathbb{R}^d.$$
Observe that $J(\cdot,\cdot)$ is the Jacobian matrix   of the constraint  $(\ref{equa5A})$ 
where $R(t)$ is the projector matrix along $\text{Im} A(t)$ with $A(t)R(t)=0, ~t\in \left[0,T\right] $ \footnote{see \cite{Renate2003} for more explanation}.
\newline\\
\begin{definition} \label{def2a}
	The Jacobian $J(\cdot,\cdot)$ possesses a global inverse if there exists a positive constant $N\neq 0$ such that 
	$$ \det( J(t,Y))= N,\;\;\; \text{ for all}\;\;  Y\in \mathbb{R}^d\;\;\; t\in \left[0, T\right]. $$
\end{definition}

Throughout this work,  $(\Omega, \mathcal{F},\mathbb{P})$ denotes a complete probability space with the natural 
filtration $\left( \mathcal{F}_t\right)_{t\geq 0} $. We define,  $\left\|Y \right\|^2=\sum_{i=1}^{n}\left| Y_i\right|^2  $ , for any vector $Y\in \mathbb{R}^d$ and   $\left| B \right|^2_F=\sum_{i=1}^{d}\sum_{j=1}^{d_1}\left|b_{i,j}\right| ^2 $ is the Frobenius norm, for any matrix $B=(b_{i,j})_{i,j=1}^{d, d_1}$.  We denote also by $\left\|\cdot\right\|_{\infty}$ the supremum norm for continuous functions. Moreover, for $A:t\mapsto  A(t),~t\in \left[ 0,T\right]$, we have  $\left\| A\right\|_{\infty}:=\max_{t\in \left[ 0,T\right]}\left| A(t)\right| _F$. Recall that, $a\lor b=\max(a,b)$ is the maximum between a  and  b,  and  $a\land b=\min(a,b)$ is the minimum between $a$ and $b$, for $a$, $b$ $\in \mathbb{R}$ .

\subsection{Well-posedness of the SDAEs}
Our well-posedness result will be investigated under the following assumptions.
\begin{Assumption}
	\label{assum1}
	We assume that
	\begin{enumerate}
		\item [(A1.1)]
		The functions $\mu(\cdot,\cdot)$ and $g(\cdot,\cdot)$ are locally Lipschitz with respect to $Y$ i.e. for any $q > 0$,   there exists $L_q > 0$, such that:
		\begin{equation}\label{equa4}
			\left\| \mu(t,X)-\mu(t,Y)\right\|\lor	\left| g(t,X)-g(t,Y)\right|_F \leq L_q \left\|Y- X\right\|,~t\in \left[0,T\right]~~~~~~~~~~~~~~~~~~~~
		\end{equation}
		for all $ X, Y\in \mathbb{R}^d$ with $ \left\|X \right\|\lor \left\|Y \right\|<q $.\\
		\item [(A1.2)]  The functions $\mu(\cdot,\cdot)$ and $g(\cdot,\cdot)$ satisfy  the Khasminskii  condition in  \cite[Theorem 3.5]{Khasminskii},  that is   there exists  a non-negative
		function $V(\cdot, \cdot)\in  C^{1;2}(\left[0,T\right]\times \mathbb{R}^d)$ and the  constants $M>0$  such that
		\begin{align}
			LV(t,Y)&\leq MV(t,Y),~t\in \left[0,T\right], Y\in \mathbb{R}^d, \label{eqp1}~~~~~~~~~~~~~~~~~~~~~~~~~~~~~~~~~~~~~~~~~~~~~~~~~~~~\\
			V_q(t)~~~~&=\inf_{\left\| Y\right\|>q } V(t,Y)\to \infty ~~ as~~ q\to \infty,\label{eqp3}
		\end{align}
		for any $q>0$, 
		and $ L$ is the differential operator given by
		$$L=\frac{\partial}{\partial t}+\sum_{i}(A^-(t)\mu(t,Y))_i\frac{\partial}{\partial Y_i}+\frac{1}{2}\sum_{i,j}(A^-(t)g(A^-(t)g)^T)_{ij}(t,Y)\frac{\partial^2}{\partial Y_i \partial Y_j}.~~~~~~~~$$
		Here $A^-(\cdot)$ is the pseudo-inverse matrix associated to the matrix $A(\cdot)$.\\
		
		\item [(A1.3)] The singular matrix $A(t),~t \in \left[0,T \right] $ and the pseudo-inverse matrix $A^-(t) ~ t\in \left[0,T \right]$ are such that the function $P(t)=A^-(t)A(t)$ is differentiable and  satisfy the relation 
		\begin{equation}\label{eqp2}
			(A^-(t)A(t))'= 0_{d\times d},~~t\in \left[ 0,T\right].~~~~~~~~~~~~~~~~~~~~~~~~~~~~~~~~~~~~~~~~~~~~~~~~~~~~~~~~~~~~~~~~~~
		\end{equation}
		Here $ A'(t)$ represents the derivative with respect to the time $t$, and $0_{d\times d}$ is a null matrix of size $d\times d$.
	\end{enumerate}
\end{Assumption}
\begin{remark}
	Note that the Assumption (A1.2) is the general condition of the monotonicity condition used in \cite{serea2025existence}. For example for a  function   $$V(Y,t)=\left\|Y\right\|^2,~t\in [0, T], ~Y\in \mathbb{R}^d.$$ 
	We obtain exactly the same monotone condition in  \cite[Assumption 1 (A1.1) with p=2]{serea2025existence}. 
	Indeed  for $~t\in [0,T], ~Y\in \mathbb{R}^d$, we have $$ \frac{\partial V}{\partial t}(Y,t)=0,~\frac{\partial V}{\partial Y}(Y,t)=2Y,~ \frac{\partial ^2V}{\partial Y_i \partial Y_j}(Y,t)=2I_d,$$ $I_d$ is the identity matrix.
	From (A1.2) and by using  the fact that $P(t)=A^-A(t)$  and  $A^-AA^-=A^-$, we have
	\begin{eqnarray*}
		LV(t,Y)&=&\frac{\partial V(t,Y)}{\partial t}+\sum_{i}(A^-(t)\mu(t,Y))_i\frac{\partial V(t,Y)}{\partial Y_i}\\
        &+&\frac{1}{2}\sum_{i,j}(A^-(t)gA^-(t)g^T)_{ij}(t,Y)
		\frac{\partial ^2V}{\partial Y_i \partial Y_j}(t,Y)\\
		&=&0+\sum_{i}(A^-(t)\mu(t,Y))_i2Y_i+\sum_{i,j}(A^-(t)gA^-(t)g^T)_{ij}(t,Y)\\
		&=&2\left\langle (P(t) Y)^T,A^-(t)\mu(t,Y)\right\rangle +\left\|A^-(t)g(t,Y\right\|^2 \\
		&\leq & MV(t,Y)=M\left\|Y\right\|^2,~t\in [0,T], ~Y\in \mathbb{R}^d.
	\end{eqnarray*}		
	Therefore we have 
	\begin{align}\label{we}
		\left\langle (P(t)Y)^T,A^-(t)\mu(t,Y)\right\rangle +\frac{1}{2}\left\|A^-(t)g(t,Y)\right\|^2\leq & M(1+\left\|Y\right\|^2),~t\in [0,T], ~Y\in \mathbb{R}^d.
	\end{align}
\end{remark} 
\begin{example}
	Let us  consider the functions $$\mu(t,y)=y\ln(y^2+1)\;\text{ and}\; g(t,y)=\sqrt{2}\ln(y^2+1),~y\in \mathbb{R},~ t\in [0,10].$$\\
	We can easily check that the functions $\mu$ and $g$ are locally Lipschitz. But they do not satisfy the monotone condition \eqref{we}. Indeed we have
	
		\begin{align*}
		h(y)=	\frac{\left\langle y,\mu(t,y)\right\rangle +\frac{1}{2}\left\|g(t,y)\right\|^2}{(y^2+1)}&=\frac{y^2\ln(y^2+1) +(\ln(y^2+1))^2}{(y^2+1)}~~~~~~~~~~~~~~~~~~~~~~~~~~\\
		&=\ln(y^2+1)\left(\frac{y^2}{(y^2+1)}+\frac{ln(y^2+1)}{(y^2+1)}\right),~y\in \mathbb{R}.
	\end{align*}
	
	We can  easily prove that $ \underset {y \to +\infty}{\lim}h(y)=+\infty$. 
	
	However, the functions $\mu$ and $g$ satisfy the general conditions defined in Assumption (A1.2) for $V(t,y)=\ln(y^2+1)$.  Indeed we have
		\begin{eqnarray*}
		\frac{\partial V(t,y)}{\partial t }= 0;~~~~	\frac{\partial V(t,y)}{\partial y }=\frac{2y}{y^2+1},\\
		\frac{\partial^2 V(t,y)}{\partial y^2 }=\frac{2(y^2+1)-4y^2}{(y^2+1)^2}=\frac{2(-y^2+1)}{(y^2+1)^2},\;\; y\in \mathbb{R}.
	\end{eqnarray*}
	Here the operator  $L$ is defined  as $$L:=\frac{\partial}{\partial t}+\sum_{i}\mu_i(t,y)\frac{\partial}{\partial y_i}+\frac{1}{2}\sum_{i,j}({gg^T}_{ij}(t,y)\frac{\partial^2}{\partial y_i \partial y_j}.~~~~~~~~~~~~~~~~~~~~~~~~~~~~~~~~$$
	We therefore have 
	\begin{align*}
		LV(t,y)&=  y\ln (y^2+1)\frac{2y}{y^2+1}+(\ln(y^2+1))^2\frac{2(-y^2+1)}{(y^2+1)^2}~~~~~~~~~~~~~~~~~\\
		&\leq \frac{2y^2\ln(y^2+1)}{y^2+1}+\frac{2(\ln(y^2+1))(-y^2+1)}{y^2+1}\\
		&\leq \frac{2y^2\ln(y^2+1)}{y^2+1}+\frac{2\ln(y^2+1)}{y^2+1}\\
		&\leq 2\ln(y^2+1)\frac{^2y+1}{y^2+1}\\
		&\leq 2\ln(y^2+1),\\
		&\leq 2V(t,y),~y\in \mathbb{R}.
	\end{align*}
\end{example}
\begin{remark}\label{rem1}
	Note that when the matrix A($\cdot$) is non-singular the projector $A^-(\cdot)A(\cdot)$ is the identity matrix and  $(A^-(t)A(t))'= 0_{d\times d},~~t\in \left[ 0,T\right] $.
\end{remark}
\begin{example}
	We can present some matrices that satisfy the Assumption (A1.3)
	\begin{itemize}
		\item[a)] Let $A(t)=\begin{pmatrix}
			0& 0  \\
			t^2+1 & 0
		\end{pmatrix}$ then $A^-(t)=\begin{pmatrix}
			0& \frac{1}{t^2+1} \\
			0 & 0
		\end{pmatrix}$ and \\
		$P(t)=A^-(t)A(t)=\begin{pmatrix}
			1 & 0 \\
			0 &0
		\end{pmatrix} $ and $P'(t)=\begin{pmatrix}
			0 & 0 \\
			0 &0
		\end{pmatrix}, t\in \left[ 0,T\right]$
		
		\item[b)] Let $A(t)=\begin{pmatrix}
			-t& 0  \\
			t & 0
		\end{pmatrix}$ then $A^-(t)=\begin{pmatrix}
			-\frac{1}{2t}& \frac{1}{t^2+1} \\
			0 & 0
		\end{pmatrix}$ and  \\
		$P(t)=A^-(t)A(t)=\begin{pmatrix}
			1 & 0 \\
			0 &0
		\end{pmatrix} $ and $P'(t)=\begin{pmatrix}
			0 & 0 \\
			0 &0
		\end{pmatrix}, t\in \left] 0,T\right]$
		\item [c)] Let $A(t)=\begin{pmatrix}
			-t& 0 &0 \\
			0 & 0&0\\
			t & 0&1
		\end{pmatrix}$ then $A^-(t)=\begin{pmatrix}
			-\frac{1}{2t}&0&0 \\
			0 & 0&0\\
			1 & 0&1
		\end{pmatrix}$ and\\
		$P(t)=A^-(t)A(t)=\begin{pmatrix}
			1 & 0&0 \\
			0 &0&0\\
			0 &0&1
		\end{pmatrix} $ and $P'(t)=\begin{pmatrix}
			0 & 0 &0\\
			0 & 0 &0\\
			0 & 0 &0
		\end{pmatrix}, t\in \left] 0,T\right]$.
	\end{itemize}
\end{example}

The first main result of this section is given in the following theorem.
\newline\\
\begin{theorem}\label{theo1} 
	Assume that the equation \eqref{equa1a}  is an index 1 SDAEs and the Jacobian matrix $J(\cdot,\cdot)$   possesses a global inverse for any $Y\in \mathbb{R}^d$. We also assume that the conditions (A1.1), (A1.2), and (A1.3) in Assumption \ref{assum1} are satisfied.
	Then a unique process $Y(\cdot)$  solution of \eqref{equa1a} exists. 
\end{theorem}
For the proof of this theorem, we need some preliminary results.
\subsubsection{Preliminary results for Theorem \ref{theo1}}
This section introduces the key lemmas  to prove the first main result  in Theorem \ref{theo1}.
\begin{lemma}\label{lem1} Let us consider the following functions  $\hat{f}: \left[ 0,T\right]\times\mathbb{R}^d \to \mathbb{R}^d $, and $\hat{g}: \left[ 0,T\right]\times\mathbb{R}^d\to \mathbb{R}^{d\times d_1}$ define by 
	\begin{eqnarray}
		\label{efuns}
		\left\{\begin{array}{l}
		\hat{f}(t,U)= P'(t)(U(t)+\hat{v}(t,U(t)))+A^-(t)\mu(t, U(t)+\hat{v}(t,U(t))) \label{f}\\
		\newline\\
		\hat{g}(t,U)=A^-(t)g(t, U(t)+\hat{v}(t,U(t))),~U\in \mathbb{R}^d, ~t\in \left[ 0,T\right]\label{g},
		 \end{array}\right.
	\end{eqnarray}
	where   $P'(t)$ is the  derivative with respect to the time $t$.
	We assume that the functions $\mu(\cdot,\cdot)$ and $g(\cdot,\cdot)$ satisfied the condition (A1.1) in Assumption 1 and  the function $\hat{v}(\cdot,\cdot)$ is such that the solution given in the Theorem \ref{theo1} is in the form $Y(t)=U(t)+\hat{v}(t,U(t)),~t\in \left[0,T \right] $. Then 
	the functions $\hat{f}(\cdot,\cdot)$ and $\hat{g}(\cdot,\cdot)$ satisfied the locally Lipschtiz condition with respect to  the second  variable $u$.

\end{lemma}
\begin{proof}
	We refer the readers to Part 2 in the proof of  \cite[ Theorem 1]{serea2025existence}.
\end{proof}
\begin{lemma}\label{lem2}
	Let us reconsider the functions $\hat{f}(\cdot, \cdot)$ and $\hat{g}(\cdot, \cdot)$ defined in \eqref{efuns}.  We assume that  the condition (A1.3) in Assumption \ref{assum1} is satisfied.    If   the functions $\mu(\cdot,\cdot)$ and $g(\cdot,\cdot)$ satisfied the condition (A1.2) in Assumption \ref{assum1},  then there exists  a non-negative function $V_1(\cdot,\cdot)\in \mathcal{C}^{1,2}(\left[ 0,T\right]\times \mathbb{	R}^d )$ and a positive constant $M_1$ such that
	\begin{eqnarray}
	\left\{\begin{array}{l}
		L_1V_1(t,U)\leq M_1 V_1(t,U), ~U\in \mathbb{R}^d, t\in \left[0,T \right]\\
		\newline\\
		V_{1_{q}}(t)=\inf_{\left\| U\right\| >q}V_1(t,U)\to \infty ~as ~~q\to \infty.
	\end{array}\right.
	\end{eqnarray} 
	for any $q>0$ such that $\left\| U\right\|<q ,~U\in \mathbb{	R}^d,~~ and ~~ t\in \left[0,T \right]$. The differential  operator  $L_1$ being  defined by
	$$L_1:=\frac{\partial}{\partial t}+\sum_{i}\hat{f}_i(t,U)\frac{\partial}{\partial U_i}+\frac{1}{2}\sum_{i,j}(\hat{g}\hat{g}^T)_{ij}(t,U)\frac{\partial^2}{\partial U_i \partial U_j}.~~~~~~~~~~~~~~~~~~~~~~~~~~~~~~~~~~~$$
\end{lemma}
\begin{proof}
	Let $U\in \mathbb{	R}^d,~~ and ~~ t\in \left[0,T \right]$,  we define $V_1$ as  $V_1(t,U)=V(t,U))$,  where $V$ is the function defined in  (A1.2) of  Assumption \ref{assum1}.
	By  using (A1.3) of  Assumption \ref{assum1}, we have
	\begin{align*}
		L_1V_1(t,U)&=\frac{\partial V(t,U)}{\partial t}+\sum_{i}\hat{f}_i(t,U)\frac{\partial V(t,U)}{\partial U_i}+\frac{1}{2}\sum_{i,j}(\hat{g}\hat{g}^T)_{ij}(t,U)\frac{\partial^2 V(t,U)}{\partial U_i \partial U_j}\nonumber\\
		&=\frac{\partial V(t,U)}{\partial t}+\sum_{i}\left(  A^-(t)\mu(t,U(t)+\hat{v}(t,U(t)))\right) _i\frac{\partial V(t,U)}{\partial U_i}\nonumber\\
		&+\frac{1}{2}\sum_{i,j}\left(A^-(t)g(t,U(t)+\hat{v}(t,U(t)))(A^-(t)\right.\\
        &\left. \times g(t,U(t)+\hat{v}(t,U(t))))^T\right)_{ij}\frac{\partial^2 V(t,U)}{\partial U_i \partial U_j}\\
		&=LV(t,U)\\
		& \leq MV(t,U)=MV_1(t,U), ~ t\in \left[0,T \right], ~U\in \mathbb{R}^d.
	\end{align*}
	We take $M_1=M$, and the proof is completed.
\end{proof}
At this point, we have the necessary results of the  proof of Theorem \ref{theo1}.\\
\subsubsection{Proof of the first main result  in Theorem \ref{theo1}}

In this section, we present substantially the transformation  of the SDAEs \eqref{equa1a} into regular SDEs \eqref{equa7A} with algebraic constraints. We refer the readers to \cite{serea2025existence} for details. We also establish that the coefficients of the new SDEs and the SDAEs  satisfy the  same conditions.  Finally, we use the classical results from SDEs to conclude the proof.
For  any vector Y, let us consider the following transformation
\begin{equation}\label{xx}
	Y=P(t)Y+Y-P(t)Y=P(t)Y+(I-P(t))Y=P(t)Y+Q(t)Y,
\end{equation}  where $P(\cdot)$  is a projector matrix and I the identity matrix and $t\in \left[0,T\right]$.
Let $$U=P(t)Y \text{ and } v=Q(t)Y .$$ We write the process $Y(\cdot)$ as     $$Y(t)=U(t)+v(t),~t \in \left[0,T \right].$$
Note that the choice of $P(\cdot)$ is such that the $$	\text{Ker}A(t)=\text{Im}(I-P(t)),~ t\in \left[0,T\right].$$ Let  $R(\cdot)$ another projector such that $$R(t)A(t)=0,~t\in \left[0,T\right].$$ The existence of the projectors $R(\cdot)$ and $P(\cdot)$ are given in \cite[Proposition 1]{serea2025existence}.
We can deduce from this notation the constraint equation by multiplying the equation \eqref{equa1a} by the projector $R(t),~ t\in \left[0,T\right]$, and we have
\begin{equation}\label{equa5A}
	A(t)v(t)+R(t)\mu(t,U(t)+v(t))= 0_{d\times d},                                 ~~ t\in\left[0,T \right]. ~~~~~~~~~~~~~~~~~~~~~~~~~~~~~~~~~~~~
\end{equation} 
Because the Jacobian matrix of \eqref{equa5A} is invisible then there exists a unique implicit function note $\hat{v}(\cdot,\cdot)$ solution of \eqref{equa5A} and the process $Y(\cdot)$ becomes: $$Y(t)=U(t)+\hat{v}(t,U(t)), ~~t\in \left[0,T \right]. $$

From \cite[eq (16) ]{serea2025existence},  inherent regular SDE (under P) associated to the
SDAE \eqref{equa1a} is

\begin{eqnarray}\label{equa7A}
	dU(t)&=&\left[ P'(t)(U(t)+\hat{v}(t,U(t)))+A^-(t)\mu(t,U(t)+\hat{v}(t,U(t)))\right] dt\nonumber\\
	\newline
	\nonumber\\
	&&+ \left[ A^-(t)g(t,U(t)+\hat{v}(t,U(t))\right]dW(t), \quad t\in \left[0,T \right] .
\end{eqnarray}



Let us observe that in equation \eqref{equa7A}, the variable $U(\cdot)$ is the new unknown variable.

In  \cite[Theorem 3.5]{Khasminskii}, it was demonstrated that when the coefficients of the Stochastic Differential Equations (SDEs) satisfy the conditions (A1.1) and (A1.2) in Assumption \ref{assum1},  then the SDE  has a unique solution. We establish in Lemma \ref{lem1} that the coefficients of equation \eqref{equa7A} also satisfy the condition (A1.1) in Assumption \ref{assum1}. In addition in Lemma \ref{lem2}, we have demonstrated that these coefficients satisfy the requirements of \eqref{eqp1} and \eqref{eqp3}.
As a result, a unique process $Y(\cdot)$  solution of the SDAEs  \eqref{equa1a} exists. The proof is therefore completed.

In realistic applications, there is no hope to have an explicit form of the exact solution $Y(t)$  of \eqref{equa1a}. Let us provide a numerical algorithm to approximate $Y(t)$. The result is for general domain  $D\subseteq\mathbb{R}^d$.
Our numerical  result is  studied under the following assumption 
\newpage
\begin{Assumption}
	\label{assum2}
	We assume that
	\begin{itemize}
		\item [(A2.1)] There exists an increasing sequence of bounded domains $\left\lbrace D_q\right\rbrace_{q=1}^{\infty} $ such that\\ $\cup_{q=1}^{\infty}D_q=D$, 
		$\sup_{Y\in D_q}\left\|f(t,Y) \right\|\leq M_q,  ~\text{and}~\sup_{Y\in D_q}\left|g(t,Y) \right|_F^2\leq M_q ,$  for every $q>0$,  $t\in \left[0,q\right]$, where $M_q$ is a constant.
		\item [(A2.2)]  $\mathbb{P}(\zeta\in D)=1.$

		\item [(A2.3)] The matrices $ A^-(t),~ B(t),~ P'(t), $ $D_h(t)=(A(t)-\theta B(t))^{-1}$  are  bounded and Lipschitz with the constant $K_1$,  for  $\theta\in (0,1)$ and $t\in \left[0,T \right]$. 
	\end{itemize}
\end{Assumption}
	\begin{remark}\label{rem34}
	
	 Note that  if $D=\mathbb{R}^d$, we can take  $D_q=\{Y\in \mathbb{R}^d, ~~\|Y\|<q\},\;\;\;q\geq 1,$  then  (A2.1) of 
	 Assumption \ref{assum2} becomes
		$$ \sup_{\|Y\|<q}\{\|f(t,Y)\|+|g(t,Y)\|_F^2\}\leq 2M_q, \text{ for every  } t\in [0,T] \text{ and } q\geq 1.$$
		In addition if we take  $V(t,Y)=(1+\|Y\|^2)e^{-Mt} ;~M>0$,  then (A1.2) of   Assumption \ref{assum1} gives   the  standard monotonicity  condition 
		$$\langle (A^{-}(t)A(t) Y)^T, A^-(t)f(t,Y)\rangle+ \frac{1}{2}|A^-(t)g(t,Y)\|_F^2\leq M(1+\|Y\|^2), ~ t\in [0,T], ~Y\in \mathbb{R}^d .$$
        Note that $A^-=A^-AA^-$.  
	\end{remark}
%

		\begin{remark}\cite{gyongy1998note}
        \label{def4}
		Note that a function $f$ defined on $[0,T]\times D$ and satisfying (A2.1) of Assumption \ref{assum2}  is locally Lipschitz in $ D$  if there exists  a bounded measurable  function  $f_q: [0,T]\times \mathbb{R}^d\to \mathbb{	R}^d$ such that $f_q(t,Y)=f(t,Y)$\footnote{We say that $f$ and $f_q$ agree on $D_q$.}  for  $ t\in [0,T],~Y\in D_q,\,\, q\geq 1$ and 
				\begin{eqnarray}
				\left\{\begin{array}{l}
				\|f_q(t,Y)\|	\leq L_q, ~Y\in \mathbb{R}^d , t\in \left[0,T \right]\\
					\newline\\
						\|f_q(t,X)-f_q(t,Y)\|\leq L_q\|X-Y\|, ~X,Y\in \mathbb{R}^d , t\in \left[0,T \right].
				\end{array}\right.
			\end{eqnarray} 
		\end{remark}

\section{Numerical method and its mathematical analysis}
Our objective is to develop and analyze a numerical method to approximate the solution of the stochastic differential algebraic equations (SDAEs) defined in \eqref{equa1a} based on semi-implicit Euler scheme.\\
Let us construct our numerical scheme for the approximation of the solution for the equation \eqref{equa1a}.\\
Let $T>0$, $n\in\mathbb{N}$  and $h=\frac{T}{n}$.
By applying the semi-implicit Euler method in equation \eqref{equa1a},  we obtain the following scheme given for  $i=0,1,2,...,n-1$ and $Y(t_i^n)\approx X_i$ by 
\begin{align}\label{equa2b}
	A(t_i^n)(X_{i+1}-X_i)&= hB(t_i^n)X_{i+1}+f(t_i^n, X_{i})h+g(t_i^n, X_i) (W(t_{i+1}^n)-W(t_i^n)),\nonumber\\
    & X_0=\zeta.
\end{align}

We find it convenient to use the continuous-time approximation, and hence we define $\bar{X}^n(t),~t\in \left[t_i^n,t_{i+1}^n \right] $  by
\begin{eqnarray}
\label{neweq}
	A(t_i^n)(\bar{X}^n(t)-X_i)&= (t-t_i^n)B(t_i^n)\bar{X}^n(t)+f(t_i^n, X_i)(t-t_i^n)\nonumber\\
    &+g(t_i^n, X_i)(W(t)- W(t_i^n)),~t\in \left[t_i^n,t_{i+1}^n \right] 
\end{eqnarray}
with   $\bar{X}^n(t_i^n)=X_{i} $.\\

Note that the equation \eqref{neweq} can be written as
\begin{eqnarray*}
\label{neweq1}
	A(k_n(t_i^n))(\bar{X}^n(t)-X_i)= (t-t_i^n)B(k_n(t_i^n))\bar{X}^n(t)+f(k_n(t_i^n), \bar{X}^n(k_n(t_i^n)))(t-t_i^n)\\+g(k_n(t_i^n), \bar{X}^n(k_n(t_i^n))(W(t)- W(t_{i}^n), 
\end{eqnarray*}
with  $k_n(t)=t_i^n=\frac{i T}{n} ,~ t\in \left[ t_i^n ,t_{i+1}^n\right)$.\\

Note that $\bar{X}^n(t)$ is only defined in $\left[t_i^n,t_{i+1}^n \right]$. To be more rigorous, we define the  continuous-time approximation in the interval $[0,T]$ by 
\begin{eqnarray*}
\label{neweq2}
\bar{X}_n(t) = \bar{X}^n(t)\,\,\, \text{ if}\,\,\,t\in \left[t_i^n,t_{i+1}^n \right].
\end{eqnarray*}
To ease  the notation, in the sequel of this paper, we will identify $\bar{X}_n(t)$ by $\bar{X}^n(t)$.
In our analysis, it will  be more natural to work with the  following equivalent  integral form  given for $ t\in \left[t_i^n,T \right]$ by 
\begin{eqnarray}
	\label{equa3b}
	A(k_n(t))(\bar{X}^n(t)- \bar{X}^n(t_i^n))= \int_{t_i^n}^{t}B(k_n(s))\bar{X}^n(s)+f(k_n(s), \bar{X}^n( k_n(s))ds\nonumber\\
    +\int_{t_i^n}^{t}g(k_n(s), \bar{X}^n( k_n(s))dW(s) ,~~~~~~~~~~~~~~~
\end{eqnarray}
or in the differential form
\begin{eqnarray}
	\label{equa4b1} 
	A(k_n(t))d\bar{X}^n(t)=\left[ B(k_n(t))\bar{X}^n(t)+f(k_n(t), \bar{X}^n( k_n(t))\right] dt\nonumber\\ +g(k_n(t), \bar{X}^n( k_n(t))dW(t)
    ,~~~~~~~~~~\bar{X}^n(0)=\zeta.~~~~~~
\end{eqnarray}

  Our key result, which is our second main result is the pathwise convergence result given  in the following theorem.\\


\begin{theorem}\label{theo3}
	Assume that Assumptions \ref{assum1} and \ref{assum2} hold. Let $ Y(\cdot) $ and $ \bar{X}^n(\cdot), ~~n\in \mathbb{	N}$ be respectively the solution of \eqref{equa1a} and its numerical approximation  given in \eqref{equa2b}-\eqref{equa4b1}.
    Then for every $\alpha <\frac{1}{2}$  there exists a finite random variable $\beta>0$, such that
	\begin{equation}\label{equa8a}
		\sup_{t\leq T}\left\| \bar{X}^n(t)-Y(t)\right\|\leq \beta n^{-\alpha} .
	\end{equation} 
\end{theorem}
The proof of this theorem is heavy and needs more preliminary results.

\subsection{Preliminary results  for Theorem \ref{theo3}}
Our first preliminary result is given in the following lemma.
\begin{lemma}\label{lemaw}
	Let $Y(t)$ be  the  $\mathcal{F}_t$-adapted process  solution  of \eqref{equa1a}  for $t<\tau:=\inf\{t: Y(t)\notin D\}$.  Assume that the Assumption \ref{assum2}  is  satisfied,  then $\tau=\infty$ a.s.
\end{lemma}
\begin{proof}
	See \cite[ Lemma 2.2]{gyongy1998note}  for the proof of Lemma \ref{lemaw}.
\end{proof}
\begin{remark}\label{rem1}  Note that  Lemma \ref{lemaw} means  that    the  $\mathcal{F}_t$-adapted process  solution  of \eqref{equa1a}  does not ever leave $D$.
\end{remark}
The following result is also important in the proof of our second main result.
\begin{lemma}\label{lem3}
	Let $p\geq 0$ , $g\in M^2(\left[0,T \right],\mathbb{R}^{d\times d_1} )$ such that: $\mathbb{E}\int_{t_0}^{t}\left|g(s) \right|^pds<\infty $ then 
	\begin{equation*}
		\mathbb{E}\left|\int_{t_0}^{t}g(s)dW(s) \right|^p\leq \left( \frac{p(p-1)}{2}\right)^{p/2}\left( T-t_0\right)^{\frac{p-2}{2}}\mathbb{E}\int_{t_0}^{t}\left|g(s) \right|^pds .~~~~~~~~~~~~~~~~~~~~
	\end{equation*}
	\end{lemma}
\begin{proof}
	For the proof, we refer the readers to \cite[Theorem 1.7.1]{Mao2007}.
\end{proof}

The next lemma is important for the transformation from SDAEs to SDEs and constraints.
\begin{lemma}\label{lem1b}
	Let $B\in \mathcal{M}_{d\times d}(\mathbb{R})$ be a matrix and  $B: \mathbb{R}^d\to \mathbb{R}^d$ the associated linear operator. Then, there exists a projector matrix $Q$ onto $Ker(B)$ such that $Im(Q)=Ker(B)$. Moreover, we can find another projector matrix $R$ along $Im(B)$ such that $RB= 0_{d\times d}$.
\end{lemma}
\begin{proof}
	We refer the reader to \cite[Proposition 1]{serea2025existence}.
\end{proof}
By using the  previous lemma we are able to establish this preliminary result as  follows.

\begin{lemma}\label{lem2b}
	Let $B \in \mathcal{M}_{n\times n}(\mathbb{R})$ a singular matrix, $Q \in \mathcal{M}_{n\times n}(\mathbb{R})$ is a projector given by Lemma \ref{lem1b} such that $Im(Q)=Ker(B)$. Then there exists a suitable non-singular matrix $D \in \mathcal{M}_{n\times n}(\mathbb{R})$ such that $DB=I-Q=P$\footnote{Observe that $P=I-Q$ is also a projector.}.
\end{lemma}
\begin{proof}
	We refer the readers to \cite[Proposition 2]{serea2025existence}.
\end{proof}
%
The most important preliminary result is this equivalence result. 
\begin{lemma}\label{theo1b} Suppose that the matrix $(A(t_i^n)+R(t_i^n)B(t_i^n)),i,n\in\mathbb{N}$ is a non singular matrix.
	The scheme defined in equation \eqref{equa2b} is equivalent to the following scheme
    \begin{equation}\label{equa5b}
		\left \{
		\begin{array}{c c c}
			u_{i+1}-u_i =& P'(t_i^n)\left[ u_{i+1}+\hat{v}^n(t_{i+1}^n,u_{i+1})\right]h+\hat{f}(t_i^n,u_{i})h~~~~~~~~~~~~~~~~~~\\
            \\
            &+A^-(t_i^n) B(t_i^n)\left[u_{i+1}+\hat{v}^n(t_{i+1}^n,u_{i+1})\right]h+\hat{g}(t_i^n,u_{i})\Delta W_i,\\
			\newline\\
           \hat{v}^n(t_{i+1}^n,u_{i+1})=&-(A(t_i^n)+R(t_i^n)B(t_i^n))^{-1}\left(R(t_i^n)B(t_i^n)u_{i+1}\right)~~~~~~~~~~~~~~\\
                      \newline\\
                      &-(A(t_i^n)+R(t_i^n)B(t_i^n))^{-1}f_1(t_i^n,u_{i})~~~~~~~~~~~~~~~~~~~~~\\
             \\
               u_0=&P(0)\zeta, ~~~~~~~~~~~~~~~~~~~~~~~~~~~~~~~~~~~~~~~~~~~~~~~~~~~~~~~~~\\
          \hat{v}^n(0,u_0)=&Q(0)\zeta,~~~~~~~~~~~~~~~~~~~~~~~~~~~~~~~~~~~~~~~~~~~~~~~~~~~~~~~~~\\
               \\
			X_{i+1}=&u_{i+1} +\hat{v}^n(t_{i+1},u_{i+1}),	~ i=0,1,...,n-1,~n\in\mathbb{N},~~~~~
		\end{array}\right. 
	\end{equation}
	where the  matrix $A^{-}(t)$ is the pseudo-inverse matrix of the matrix $A(t)$,   $P(t)$  and $R(t)$ are projectors matrix associated to $A(t)$,  $t\in \left[0,T \right] $ and 
     \begin{equation}\label{chapeau}
		\left \{
        \begin{array}{c c c}
        \hat{f}(t_i^n, u_n)&=A^-(t_i^n)f(t_i^n,u_{i}+\hat{v}^n(t_i^n,u_{i}))\\
        \\
        \hat{g}(t_i^n, u_n)&=A^-(t_i^n)g(t_i^n,u_{i}+\hat{v}^n(t_i^n,u_{i}))\\
        \\
        f_1(t_i^n,u_n)&=R(t_i^n)f(t_i^n,u_{i}+\hat{v}^n(t_i^n,u_{i})).
        \end{array}\right. 
	\end{equation}
    Note that the functions $\hat{f} \text{ and } f_1$ satisfy the same conditions as the functions $f$ and the function $\hat{g} $ satisfied the same conditions as the function $g$  (see for example \cite{serea2025existence} ).
\end{lemma}
\begin{proof}
	The idea here is to transform the numerical scheme for SDAEs to a numerical scheme for SDEs with the associated constraints. 
	In this proof, we will use the continuous form associated to scheme \eqref{equa2b} defined in  \eqref{equa4b1} by	
	\begin{eqnarray*} 
	A(k_n(t))d\bar{X}^n(t)=\left[ B(k_n(t))\bar{X}^n(t)+f(k_n(t), \bar{X}^n( k_n(t))\right] dt\\ 
    +g(k_n(t), \bar{X}^n( k_n(t))dW(t)
    ,~~~~~~~~~~\bar{X}^n(0)=\zeta~~~~~~
\end{eqnarray*}
    
	Recall that Lemma \ref{lem1b} provides the existence of two projectors $Q(t)$ and $R(t)$ such that  $ImQ(t)=KerA(t)$, and $R(t)A(t)= 0_{d\times d}$, $t\in \left[0, T\right] $.
	 From equation \eqref{xx} or from \cite{serea2025existence} and \cite{Renate2003} the solution $\bar{X}^n(\cdot)$ can be written as 
	\begin{eqnarray}
	\label{equa6b}
	\left \{
	\begin{array}{l}
		\bar{X}^n(t)=P(k_n(t))\bar{X}^n(t)+Q(k_n(t))\bar{X}^n(t)=\bar{u}^n(t)+\bar{v}^n(t),\\
		\newline\\
		 \bar{u}^n(t)\in \text{Im}P(k_n(t)), \quad \bar{v}^n(t)\in \text{Im}Q(k_n(t)) ,\quad t\in\left[0,T \right].
	\end{array}\right. 
       \end{eqnarray}
	Let us find the constraints equation associated to our scheme \eqref{equa4b1}.
	Let us multiply the equation (\ref{equa4b1})  by the matrix $R(\cdot)$. Recall that $R(t)g(t, \bar{X}^n(t))=0,  ~t\in \left[0, T \right] $\footnote{see  the definition of index 1 SDAEs (Definition \ref{def2b})}.\\ 
	
	Moreover, we have $A(k_n(t))\bar{v}^n(t)=0$ for  $\bar{v}^n(t)\in \text{ImQ}(k_n(t))=\text{Ker}A(k_n(t)),~ t\in \left[ 0,T\right] $. Consequently, we obtain the following constraints
	\begin{equation}\label{equa7b}
		A(k_n(t))\bar{v}^n(t)+R(k_n(t))B(k_n(t))(\bar{u}^n(t)+\bar{v}^n(t))+R(k_n(t))f(k_n(t),\bar{u}^n(k_n(t))+\bar{v}^n(k_n(t)))=0. ~~~~~~~~~~~~~~~~~~~~~~~~~~
	\end{equation}
	From equation \eqref{equa7b} and because $A(t)+R(t)B(t), ~t\in \left[0,    T\right]$ is a non singular matrix as  we have assumed in  Lemma \ref{theo1b}, we have
	\begin{align}\label{equa8b}
		\bar{v}^n(t)=&\hat{v}^n(t,\bar{u}^n(t))\nonumber\\
        =&-(A(k_n(t))+R(k_n(t))B(k_n(t)))^{-1}\left[R(k_n(t))B(k_n(t))\bar{u}^n(t)\right.\nonumber\\
        &\left. +R(k_n(t))f(k_n(t),\bar{u}^n((k_n(t))+\hat{v}^n(k_n(t),\bar{u}^n((k_n(t)))) \right] , ~t\in \left[ 0,T\right]. \nonumber\\
        =&-(A(k_n(t))+R(k_n(t))B(k_n(t)))^{-1}\left[R(k_n(t))B(k_n(t))\bar{u}^n(t)\right.\nonumber\\
        &\left. +f_1(k_n(t),\bar{u}^n((k_n(t))) \right] , ~t\in \left[ 0,T\right]. 
	\end{align}
    Remember that $f_1 =Rf$ is defined in \eqref{chapeau}.\\
    
	Consequently, equation \eqref{equa7b}  is the constraint equation associated to the SDAEs, and the solution $\bar{v}^n(\cdot)$ is given by the implicit function $\hat{v}^n(\cdot, \bar{u}^n(\cdot))$ given in \eqref{equa8b}
	and the equality $(\ref{equa6b})$ becomes
	\begin{equation}\label{equa9b}
		\bar{X}^n(t)=\bar{u}^n(t)+\hat{v}^n(t,\bar{u}^n(t)),\quad  [0,T].~~~~~~~~~~~~~~~~~~~
	\end{equation}
	Note that the implicit function $\hat{v}^n$ is globally Lipschitz with respect to the second variable  $\bar{u}^n$. Indeed  we can just prove that the $\left\|\frac{\partial \hat{v}^n}{\partial u}(t,u)\right\|\leq L_{\hat{v}}$ (see for example \cite[Part 2 in the proof of Theorem 1]{serea2025existence} for more explanation). Note also that $L_{\hat{v}}$ is the Lipschitz constant independent to $n$.

	Let  us find now the stochastic differential equation (SDEs) associated to the SDAEs \eqref{equa2b}.\\
	
	Let us multiply the equation $(\ref{equa4b1})$ by the projector matrix $I-R(\cdot)$. Using the fact that $R(t)g(t,X(t))=0$,  we have
	\begin{align}\label{equa10b}
		A(k_n(t))d\bar{X}^n(t)&=(I-R(k_n(t)))\left[ B(k_n(t))\bar{X}^n(t)+f(k_n(t),\bar{X}^n(k_n(t)))\right] dt\nonumber\\
        &+g(k_n(t),\bar{X}^n(k_n(t)))dW(t), ~~ t\in [0,T].
	\end{align}
	
	Let $D(\cdot)$ be a non-singular matrix provided by Lemma \ref{lem2b}. Then we have  $D(t)A(t)=P(t)$. Let  $A^-(t)=D(t)(I-R(t))$ is the  pseudo inverse of a matrix A(t) with $A^-(t)A(t)=P(t)$ and $A(t)A^-(t)=(I-R(t))$ for all $ t\in [0,T]$.\\ 
	Multiplying $(\ref{equa10b})$ by the matrix $A^-(\cdot)$   we  obtain 
	\begin{align}\label{equa111b}
		P(k_n(t))d\bar{X}^n(t)&=A^-(k_n(t))(I-R(k_n(t)))\left[ B(k_n(t))\bar{X}^n(t)+f(k_n(t),\bar{X}^n(k_n(t)))\right] dt\nonumber\\
        &+A^-(k_n(t))g(k_n(t),\bar{X}^n(k_n(t)))dW(t), ~~ t\in [0,T].
	\end{align}
    
	Note that $$A^-(t)(I-R(t))=D(t)(I-R(t))(I-R(t))=D(t)(I-R(t))=A^-(t)$$ for all $t\in \left[0,T\right].$
	We continue by defining $$h(t,\bar{X}^n)=P(k_n(t))\bar{X}^n\,\,\, t\in \left[0,T \right]$$ and $\bar{X}^n\in \mathbb{R}^d$. Observe that $h(\cdot,\bar{X}^n)$ is a It\^o process wherever $\bar{X}^n(\cdot)$  is a It\^o  process.\\  
	Because   $P(\cdot)$ is an continuously differentiable projector,  by  It\^o formula, we obtain
	\begin{align*}
		dh(t,\bar{X}^n)&=P'(k_n(t))\bar{X}^ndt+P(k_n(t))d\bar{X}^n,~t\in \left[0,T\right] .~~~~~~~~~~~~~~~~~~~~~~~~~~~~~~~~~~~~~~~~~~~~~~~~~~~~~~
	\end{align*}
	Consequently,
	\begin{align}\label{equa12b}
		P(k_n(t))d\bar{X}^n(t)&= d(P(k_n(t))\bar{X}^n(t))-P'(k_n(t))\bar{X}^n(t)dt,~t\in \left[0,T \right].~~~~~~~~~~~~~~~~~~~~~~~~~~~~~~~~~~
	\end{align}
	Combining the equations $\eqref{equa111b}$ and $\eqref{equa12b}$ we obtain
	\begin{align}\label{equa14b}
		d(P(k_n(t))\bar{X}^n(t))&=P'(k_n(t))\bar{X}^n(t)dt+A^-(k_n(t))B(k_n(t))\bar{X}^n(t)dt\nonumber\\
        &+A^-(k_n(t))f(k_n(t),\bar{X}^n(k_n(t)))dt\nonumber\\
        &+A^-(k_n(t))g(k_n(t),X(k_n(t)))dW(t), ~~
		t\in \left[0,T \right].
	\end{align}
	We recall  that by equations \eqref{equa9b} and \eqref{equa6b}  $P(k_n(t))\bar{X}^n(t)= \bar{u}^n(t)$ and $\bar{X}^n(t)=\bar{u}^n(t)+\hat{v}^n(t,\bar{u}^n(t)), ~t\in  \left[ 0,T\right]$. Therefore, we can rewrite the equation \eqref{equa14b} as follows
	\begin{eqnarray}
	\label{equa16ba}
		d\bar{u}^n(t)&=&P'(k_n(t))[\bar{u}^n(t)+\hat{v}^n(t,\bar{u}^n(t))]dt\nonumber\\
        \nonumber\\
        &+&A^-(k_n(t))B(k_n(t))[\bar{u}^n(t)+\hat{v}^n(t,\bar{u}^n(t))]dt\nonumber \\
        \nonumber\\
		&+ & (A^-(k_n(t))f(k_n(t),\bar{u}^n(k_n(t))+\hat{v}^n(k_n(t),\bar{u}^n(k_n(t))))dt\nonumber\\
        \nonumber\\
        &+&A^-(k_n(t))g(k_n(t),\bar{u}^n(k_n(t))+\hat{v}^n(k_n(t),\bar{u}^n(k_n(t))))dW(t),
	\end{eqnarray}
     This means that from \eqref{chapeau} we have
     	\begin{eqnarray}
	\label{equa16b}
		d\bar{u}^n(t)&=&P'(k_n(t))[\bar{u}^n(t)+\hat{v}^n(t,\bar{u}^n(t))]dt\nonumber\\
        \nonumber\\
        &+&A^-(k_n(t))B(k_n(t))[\bar{u}^n(t)+\hat{v}^n(t,\bar{u}^n(t))]dt\nonumber \\
        \nonumber\\
		&+ & \hat{f}(k_n(t),\bar{u}^n(k_n(t)))dt+\hat{g}(k_n(t),\bar{u}^n(k_n(t)))dW(t).
	\end{eqnarray}
	The equation \eqref{equa16ba} is called the inherent regular SDEs associated to the SDAEs \eqref{equa4b1}. 
	So we can rewrite equation \eqref{equa4b1} as a following
	\begin{equation}\label{equa15ba}
		\left \{
		\begin{array}{c c c}
			d\bar{u}^n(t)&=&P'(k_n(t))[\bar{u}^n(t)+\hat{v}^n(t,\bar{u}^n(t))]dt~~~~~~~~~~~~~~~~~~~~~~~~~~~~~~~~~~~~~~~~~~~~~~~~~~~~~~~~~~~~~~~~~~\\
            \\
            &+&A^-(k_n(t))B(k_n(t))[\bar{u}^n(t)+\hat{v}^n(t,\bar{u}^n(t))]dt~~~~~~~~~~~~~~~~~~~~~~~~~~~~~~~~~~~~~~~~~~~~~~~~~~~~~~~~~\\
			\\
			&+&\hat{f}(k_n(t),\bar{u}^n(k_n(t)))dt+\hat{g}(k_n(t),\bar{u}^n(k_n(t)))dW(t),~~~~~~~~~~~~~~~~~~~~~~~~~~~~~~~~~~~~~~~~~~~\\
			\\
			\hat{v}^n(t,\bar{u}^n(t))&=&-(A(k_n(t))+R(k_n(t))B(k_n(t)))^{-1}\left[R(k_n(t))B(k_n(t))\bar{u}^n(t)\right.~~~~~~~~~~~~~~~~~~~~~~~~~~~~~~~~~\\
            \\
            &+&\left.f_1(k_n(t),\bar{u}^n(k_n(t))) \right] ,              ~~~~~~~~~~~~~~~~~~~~~~~~~~~~~~~~~~~~~~~~~~~~~~~~~~~~~~~~~~~~~~~~~~~~~~~~~~\\
			\\
			\bar{u}^n(0)&=&P(0)\zeta,~~~~~~~~~~~~~~~~~~~~~~~~~~~~~~~~~~~~~~~~~~~~~~~~~~~~~~~~~~~~~~~~~~~~~~~~~~~~~~~~~~~~~~~~~~\\
            \\
			\hat{v}^n(0, \bar{u}^n(0))&=&Q(0)\zeta,~~~~~~~~~~~~~~~~~~~~~~~~~~~~~~~~~~~~~~~~~~~~~~~~~~~~~~~~~~~~~~~~~~~~~~~~~~~~~~~~~~~~~~~~~~\\
			\\
            
			\bar{X}^n(t)&=&\bar{u}^n(t) +	\hat{v}^n(t,\bar{u}^n(t))	,~t\in  \left[ 0,T\right]. ~~~~~~~~~~~~~~~~~~~~~~~~~~~~~~~~~~~~~~~~~~~~~~~~~~~~~~~~~~~~~ 
		\end{array}\right.  
	\end{equation}
	Finally, we can write equation \eqref{equa15ba} in the discrete form and obtain exactly equation \eqref{equa5b} in lemma \ref{theo1b}
	\begin{equation}\label{equa17b}
		\left \{
		\begin{array}{c c c}
			u_{i+1}-u_i &=& P'(t_i^n)\left[ u_{i+1}+\hat{v}^n(t_{i+1}^n, u_{i+1})\right] h+\hat{f}(t_i^n,u_{i})h~~~~~~~~~~~~~~~\\
			\\
			&+&A^-(t_i^n) B(t_i^n)\left[u_{i+1}+\hat{v}^n(t_{i+1}^n,u_{i+1})\right]h+\hat{g}(t_i^n,u_{i})\Delta W_i,\\
			\\
			\hat{v}^n(t_{i+1}^n,u_{i+1})&=& -(A(t_i^n)+R(t_i^n)B(t_i^n))^{-1}\left(R(t_i^n)B(t_i^n)u_{i+1}\right),~~~~~~~~~~~~~\\
            \\
            &&-(A(t_i^n)+R(t_i^n)B(t_i^n))^{-1}f_1(t_i^n,u_{i})~~~~~~~~~~~~~~~~~~~~~~~~~\\
			\\
	u_0&=&P(0)\zeta,~~~~~~~~~~~~~~~~~~~~~~~~~~~~~~~~~~~~~~~~~~~~~~~~~~~~~~~~~~~\\
          \hat{v}^n(0,u_0)&=&Q(0)\zeta,~~~~~~~~~~~~~~~~~~~~~~~~~~~~~~~~~~~~~~~~~~~~~~~~~~~~~~~~~~~\\
			\\
			X_{i+1}&=&u_{i+1} +\hat{v}^n(t_{i+1}^n,u_{i+1}),~ i=0,1,...,n-1,n\in\mathbb{N},~~~~~~~~ 
		\end{array}\right. 
	\end{equation}
    where the functions $\hat{f},~ \hat{g},~ f_1$ are defined in \eqref{chapeau}.
    This ends the proof.
\end{proof}


The following lemma  shows that the continuous numerical solution $\bar{u}_q^n$  \eqref{equa16b} is bounded.

\begin{lemma}\label{lemef}
	Assume that Assumptions \ref{assum1} and \ref{assum2} hold,  then the solution $\bar{u}_q^n$ of equation \eqref{equa16b} where the coefficients $f$ and g are replaced by the functions $f_q$ and $g_q$ respectively \footnote{See Remark \ref{def4}} satisfies the following inequality
	$$\mathbb{	E}\left\|  \bar{u}_q^{n}(s)\right\|^{2p}\leq C, ~s\in \left[0, T\right] , ~p\geq 1, ~C\text{ is a constant independent to n}. ~~~~~~~~~~~~~~~~~~~~~~~~~~~~~~~~~~~~~~$$
\end{lemma}
\begin{proof}	
	 We want to show that  for  $s\in \left[ 0,T\right]$  we  have  $\mathbb{E}\left\|  \bar{u}_q^{n}(s)\right\|^{2p}<\infty.$
	 	From \eqref{equa16b} we have
	\begin{align*}
		\mathbb{	E}\left\|  \bar{u}_q^{n}(s)\right\|^{2p}&=\mathbb{	E}\left\| \int_{0}^{t} P'(k_n(s))\left[ 	\bar{u}_q^n(s)+\hat{v}^n(s,	\bar{u}_q^n(s))\right]ds +\bar{u}_q^n(0)\nonumber\right.\\
		&\left.+\int_{0}^{t}A^-(k_n(s)) B(k_n(s))\left[	\bar{u}_q^n(s)+\hat{v}^n(s,	\bar{u}_q^n(s))\right]ds \right.  \nonumber\\
		&\left.+ \int_{0}^{t}\hat{f}_q(k_n(s),\bar{u}_q^n(k_n(s)))ds+\int_{0}^{t}\hat{g}_q(k_n(s),\bar{u}_q^n(k_n(s)))dW(s)\right\| ^{2p}.
                \end{align*}
        Using Assumption \ref{assum2} 
		\begin{align*}
		&\mathbb{	E}\left\|  \bar{u}_q^{n}(s)\right\|^{2p}
		\leq 5^{2p-1}\left( \mathbb{	E}\left\| \int_{0}^{t} P'(k_n(s))\left[ \bar{u}_q^n(s)+ \hat{v}^n(s,	\bar{u}_q^n(s))\right]ds\right\| ^{2p}\right. \\
		&\left.+\mathbb{	E}\left\| \int_{0}^{t}A^-(k_n(s)) B(k_n(s))\left[	\bar{u}_q^n(s)+\hat{v}^n(s,	\bar{u}_q^n(s))\right]ds\right\| ^{2p}+\mathbb{	E}\left\| \bar{u}^n_q(0)\right\| ^{2p}\right. \\
		&\left.+\mathbb{	E}\left\| \int_{0}^{t}\hat{f}_q(k_n(s),\bar{u}^n_q(k_n(s)))ds\right\| ^{2p}+\mathbb{	E}\left\| \int_{0}^{t}\hat{g}_q(k_n(s),\bar{u}_q^n(k_n(s)))dW(s)\right\| ^{2p}\right) \\
&\leq 5^{2p-1}\left(T^{2p-1}K_1\mathbb{	E} \int_{0}^{t}\left\| \bar{u}_q^n(s)+ \hat{v}^n(s,	\bar{u}_q^n(s))\right\|^{2p} ds+\mathbb{	E}\left\| \bar{u}^n_q(0)\right\| ^{2p}\right.\\
		&\left.+  T^{2p-1}\int_{0}^{t}\mathbb{	E}\left\|\hat{f}_q(k_n(s),\bar{u}_q^n(k_n(s)))\right\| ^{2p}ds+ C(p)\int_{0}^{t}\mathbb{	E}\left\|\hat{g}_q(k_n(s),\bar{u}_q^n(k_n(s)))\right\| ^{2p}ds\right) .
        \end{align*}
        Using the fact that
        \begin{align*}
        \left\| \bar{u}_q^n(s)+ \hat{v}^n(s,	\bar{u}_q^n(s))\right\|^{2p}&=\left\| \bar{u}_q^n(s)+ \hat{v}^n(s,	\bar{u}_q^n(s))-\hat{v}^n(s,0)+\hat{v}^n(s,0)\right\|^{2p}\\
        &\leq 3^{2p-1}(\left\| \bar{u}_q^n(s)\right\|^{2p}+ \left\|\hat{v}^n(s,	\bar{u}_q^n(s))-\hat{v}^n(s,0)\right\|^{2p}+\left\|\hat{v}^n(s,0)\right\|^{2p})\\
        &\leq 3^{2p-1}((1+L_{\hat{v}})\left\| \bar{u}_q^n(s)\right\|^{2p}+\left\|\hat{v}^n(s,0)\right\|^{2p})
        \end{align*}
        we obtain
        \begin{align}\label{equa39b}
		&\mathbb{	E}\left\|  \bar{u}_q^{n}(s)\right\|^{2p}
		\leq 5^{2p-1}T^{2p-1}K_1 3^{2p-1}K_1\int_{0}^{t}(1+L_{\hat{v}})\mathbb{	E}\left\| \bar{u}_q^n(s) \right\|^{2p} +\left\| \hat{v}^n(s,0)\right\|^{2p}ds\nonumber\\
        &+ C(p,L_q, K_1,\zeta) \nonumber\\
		&\leq 5^{2p-1}T^{2p-1}K_13^{2p-1}(1+L_{\hat{v}})\int_{0}^{t}\mathbb{	E}\left\| \bar{u}_q^n(s) \right\|^{2p} ds+C(L_{\hat{v}},\bar{u}^n_q(0),p,L_q, K_1) .
	\end{align}
	Using Gronwall's Lemma, we have,  for  $s\in \left[0,T \right]$
	\begin{align}\label{equa39bb}
		\mathbb{	E}\left\|  \bar{u}_q^{n}(s)\right\|^{2p}&\leq C(s,L_{\hat{v}},p,L_q, T,K, \zeta).
	\end{align}
    Note that we have used the fact that
    \begin{align*}
    \left\|\hat{v}^n(t,0)\right\|&=\left\|(A(k_n(t)+R(k_n(t))B(k_n(t))^{-1}f_{1q}(k_n(t),0)\right\|\\
    & \leq \left\|(A(k_n(t)+R(k_n(t))B(k_n(t))^{-1}\right\|\left\|f_{1q}(k_n(t),0)\right\|\\
    &\leq K_1^2L_q. 
    \end{align*}
    Here, the constant $K_1$ is defined in Assumption \ref{assum2}, and the constant $L_q$ is defined in Remark \ref{def4}.\\
    
    Note that $f_{1q}$ is the function that agrees with $f_1$ in $D_q$ according to Remark \ref{def4} and $f_1$ is defined in \eqref{chapeau}.
\end{proof}
The next lemma is the regularity result of our numerical continuous form \eqref{equa3b}.
\\ 
\begin{lemma}\label{theo3a}
	Let $\bar{X}^n_q(t))$ be the  solution of $\eqref{equa3b}$ where the coefficients $f(\cdot,\cdot)$ and $g(\cdot,\cdot)$ are replaced  by the functions   $f_q(\cdot,\cdot)$ and $g_q(\cdot,\cdot)$ respectively defined in the Remark \ref{def4}. Then
	for all $p\geq 2$, there exists a constant  $\hat{Z}$ independent to n such  that:
	\begin{eqnarray}\label{equa39}
		\mathbb{E}\left[\left\|\bar{X}^n_q(t)-\bar{X}^n_q(t_i^n) \right\|^p \right]\leq  \hat{Z}\left| t-t_i^n\right|^{p/2} , ~~~~~~~~~~~~~~~~ t\in [t_i^n,t_{i+1}^n),
	\end{eqnarray}
	with  $t_{i+1}^n=t_i^n+\frac{T}{n}=\frac{(i+1)T}{n},~i< n\in\mathbb{N}$.
\end{lemma}

\begin{proof} 
    The idea in the proof is to use the fact that the solution is the sum of algebraic and stochastic variables.
From   \eqref{xx},   we have
\begin{align}\label{equa40a}
	\mathbb{E}\left[\left\| \bar{X}^n_q(t)-\bar{X}^n_q(t_i^n)\right\|^p \right]&=  	\mathbb{E}\left[\left\| \bar{u}^n_q(t)+\hat{v}^n(t,\bar{u}^n_q(t))-\bar{u}^n_q(t_i^n)-\hat{v}^n(t_i^n, \bar{u}^n_q(t_i^n))\right\|^p \right]\nonumber~~~~~~~~~~~~~~~~~~~~~~~~~~~~\\
	&\leq 	\mathbb{E}\left[\left( \left\| \bar{u}^n(t)-\bar{u}^n_q(t_i^n)\right\|+L_{\hat{v}} \left\| \bar{u}^n_q(t)-\bar{u}^n_q(t_i^n)\right\|\right) ^p \right]\nonumber\\
	& \leq (1+L_{\hat{v}})^p \mathbb{E}\left[\left\| \bar{u}^n_q(t)-\bar{u}^n_q(t_i^n)\right\|^p  \right],~t\in  [t_i^n,t_{i+1}^n).
\end{align}

where $\bar{u}^n_q(\cdot)$ is a solution of the inherent SDEs \eqref{equa16b} associated to \eqref{equa10b}. \\

For all positive numbers $a,b$ and $c$, it is well known that  $$(a+b+c)^p\leq 3^{p-1}(a^p+b^p+c^p).$$
From the representation \eqref{equa16b}, for $t\in [t_i^n,t_{i+1}^n)$, we have
\begin{eqnarray*}
	\bar{u}^n_q(t)-\bar{u}^n_q(t_i^n)=\int_{t_i^n}^{t} (P'(k_n(z))+A^-(k_n(z))B(k_n(z)))(\bar{u}^n_q(z)+\hat{v}^n(z,\bar{u}^n_q(z)))dz\\
    +\int_{t_i^n}^{t}\hat{f}_q(k_n(z),\bar{u}^n_q(k_n(z)))dz
	+ \int_{t_i^n}^{t}\hat{g}_q(k_n(z),\bar{u}^n_q(k_n(z)))dW(z).
\end{eqnarray*}
We therefore have
\begin{align*}
	&\left\|	\bar{u}^n_q(t)-\bar{u}^n_q(t_i^n) \right\|^p \\
    & \leq 3^{p-1}\left[\left\|\int_{t_i^n}^{t} (A^-(k_n(z))B(z)+P'(k_n(z)))(\bar{u}^n_q(z)+\hat{v}^n(z,\bar{u}^n_q(z)))dz \right\|^p\right.\\
	&\left.+\left\|\int_{t_i^n}^{t}\hat{f}_q(k_n(z),\bar{u}^n_q(k_n(z)))dz \right\|^p\right]+3^{p-1}\left\|\int_{t_i^n}^{t}\hat{g}_q(k_n(z),\bar{u}^n_q(k_n(z)))dW(z) \right\|^p  \\
	&\leq 3^{p-1}\left|t-t_i^n \right|^{p-1} \left[ \int_{t_i^n}^{t} \left\|(A^-(k_n(z))B(k_n(z))+P'(k_n(z)))(\bar{u}^n_q(z)+\hat{v}^n(z,\bar{u}^n_q(z)))\right\|^pdz \right.\\ 
	&\left.+\int_{t_i^n}^{t}\left\|\hat{f}_q(k_n(z),\bar{u}^n_q(k_n(z)))\right\|^pdz\right] +3^{p-1}\left\|\int_{t_i^n}^{t}\hat{g}_q(k_n(z),\bar{u}^n_q(k_n(z)))dW(z) \right\|^p.
    \end{align*}
    By taking the expectation in both sides, we have
    \begin{align*}
	&\mathbb{E}\left[\left\|	\bar{u}^n_q(t)-\bar{u}^n_q(t_i^n) \right\|^p\right]\\ 
    & \leq 3^{p-1}\left|t-t_i^n \right|^{p-1}\left(  \int_{t_i^n}^{t} \mathbb{E}\left[\left\|(A^-(z)B(z)+P'(z))(\bar{u}^n_q(z)+\hat{v}^n(z,\bar{u}^n_q(z)))\right\|^p\right]dz\right.\\ 
	&\left.+\int_{t_i^n}^{t}\mathbb{E}\left[\left\|\hat{f}_q(k_n(z),\bar{u}^n_q(k_n(z)))\right\|^p\right]  dz\right)+3^{p-1}\mathbb{E}\left[\left\|\int_{t_i^n}^{t}\hat{g}_q(k_n(z),\bar{u}^n_q(k_n(z)))dW(z)\right] \right\|^p .
\end{align*}
Let us set  $M=2^p(K_1^{2p}+K_1^p)$. We obviously have  $$\max_{t\in \left[ 0,T\right] }(\left| (A^-(t)B(t)+P'(t))\right|^p_F )\leq 2^p(K_1^{2p}+K_1^p) =M.$$
From Lemma \ref{lem3}, Assumption \ref{assum2} and Remark \ref{def4}, we have
\begin{align*}
	&\mathbb{E}\left[\left\|	\bar{u}^n_q(t)-\bar{u}^n_q(t_i^n)) \right\|^p \right] 	\leq3^{p-1}M\left|t-t_i^n \right|^{p-1}  \int_{t_i^n}^{t} \mathbb{E}\left\|(\bar{u}^n_q(z)+\hat{v}^n(z, \bar{u}^n_q(z)))\right\|^pds\\
    &+3^{p-1}M\left|t-t_i^n \right|^{p-1} \int_{t_i^n}^t \mathbb{E}\left\|\hat{f}_q(k_n(z),\bar{u}^n_q(k_n(z)))\right\|^p  dz \nonumber\\
	&+3^{p-1}M\left( \dfrac{p}{2}\right) ^{p/2}(p-1)^{p/2}\left| t-t_i^n\right| ^{\frac{p-2}{2}}\int_{t_i^n}^{t}\mathbb{E}\left[\left| \hat{g}_q(k_n(z),\bar{u}^n_q(k_n(z)))\right|^p\right] dz\nonumber\\
	&\leq 3^{p-1}M\left|t-t_i^n \right|^{p-1}  \int_{t_i^n}^{t} \mathbb{E}\left\|\bar{u}^n_q(z)+\hat{v}^n(z,\bar{u}^n_q(z))-\hat{v}^n(0,z)+\hat{v}^n(0,z)\right\|^pdz  \nonumber\\
	&+ 3^{p-1}M\left|t-t_i^n \right|^{p-1}  \int_{t_i^n}^{t}L_q dz+3^{p-1}M\left( \dfrac{p}{2}\right) ^{p/2}(p-1)^{p/2}(t-t_i^n)^{\frac{p-2}{2}}\int_{t_i^n}^{t} L_qdz.
       \end{align*}
       Using the global Lipschitz on $\hat{v}^n$ with respect to the variable $\bar{u}_q^n$ we have
    \begin{eqnarray*}
	&\mathbb{E}\left[\left\|	\bar{u}^n_q(t)-\bar{u}^n_q(t_i^n)) \right\|^p \right]
	\leq 3^{2p-2}M\left|t-t_i^n \right|^{p-1}  \int_{t_i^n}^{t} (1+L_{\hat{v}})\mathbb{E}\left\|\bar{u}^n_q(z)\right\|^p+\left|\hat{v}^n(0,z)\right\|^pdz
	 \nonumber\\
	&+ 3^{p-1}ML_q\left|t-t_i^n \right|^{p}+3^{p-1}ML_q\left( \dfrac{p}{2}\right) ^{p/2}(p-1)^{p/2}\left| t-t_i^n\right| ^{\frac{p}{2}}~~~~\nonumber\\	
    &\leq 3^{2p-2}M\left| t-t_i^n\right|^{p-1}(1+L_{\hat{v}})  \int_{t_i^n}^{t} \mathbb{E}\left\|\bar{u}^n_q(z)\right\|^pdz
	+ Z_1\left| t-t_i^n\right| ^{\frac{p}{2}}.	
    \end{eqnarray*}
    Using the fact that 
    \begin{align*}
    \left\|\bar{u}^n_q(z)\right\|^p&=\left\|\bar{u}^n_q(z)-\bar{u}^n_q(k_n(z))+\bar{u}^n_q(k_n(z))\right\|^p\\
    &\leq 2^{p-1}(\left\|\bar{u}^n_q(z)-\bar{u}^n_q(k_n(z))\right\|^p+\left\|\bar{u}^n_q(k_n(z))\right\|^p),
    \end{align*}
    we have
    \begin{align*}
    \mathbb{E}\left[\left\|	\bar{u}^n_q(t)-\bar{u}^n_q(t_i^n)) \right\|^p \right] 
	& \leq 3^{2p-2}2^{p-1}M| t-t_i^n|^{p-1}(1+L_{\hat{v}})  \int_{t_i^n}^{t} \mathbb{E}\left\|\bar{u}^n_q(z)-\bar{u}^n_q(k_n(z))\right\|^pdz\nonumber\\
    &+3^{2p-2}2^{p-1}MT^{p-1}(1+L_{\hat{v}})  \int_{t_i^n}^{t}\mathbb{E}\left\|\bar{u}^n_q(k_n(z))\right\|^pdz
	+ Z_1\left| t-t_i^n\right| ^{\frac{p}{2}}
\end{align*}
with  $$ Z_1=MT^{\frac{p}{2}}(3^{2p-2}\max_{t\in [0,T]}\left(\| \hat{v}^n(t,0)\| \right)+3^{p-1}L_q )+3^{p-1}ML_{\hat{v}}\left( \dfrac{p}{2}\right) ^{p/2}(p-1)^{p/2}.$$

Using Lemma \ref{lemef} we have
\begin{align*}
	\mathbb{E}\left[\left\|	\bar{u}^n_q(t)-\bar{u}^n_q(t_i^n) \right\|^p \right] 	 & \leq 3^{2p-2}2^{p-1}MT^{p-1}(1+L_{\hat{v}})  \int_{t_i^n}^{t} \mathbb{E}\left\|\bar{u}^n_q(z)-\bar{u}^n_q(k_n(z))\right\|^pdz\nonumber\\
    & +3^{2p-2}2^{p-1}MT^{p-1}(1+L_{\hat{v}})| t-t_i^n|^{p} C(t)
	+ Z_1\left| t-t_i^n\right| ^{\frac{p}{2}}.
    \end{align*}
 Using Gr\"onwall inequality we have
    \begin{align}\label{equa41a}
	\mathbb{E}\left[\left\|	\bar{u}^n_q(t)-\bar{u}^n_q(t_i^n) \right\|^p \right] 	 &\leq Z\left| t-t_i^n\right| ^{\frac{p}{2}},
      \end{align}
          with $Z=(3^{2p-2}2^{p-1}MT^{p-1}(1+L_{\hat{v}})T^{\frac{T}{2}} C(t)+Z_1)\exp{\left(3^{2p-2}2^{p-1}MT^{p}(1+L_{\hat{v}})\right)}.$
Combining \eqref{equa40a} and \eqref{equa41a} allows to have
\begin{align*}
	\mathbb{E}\left[\left\| \bar{X}^n_q(t)-X_q(t_i^n)\right\|^p \right]& \leq (1+L_{\hat{v}})^p  Z\left| t-t_i^n\right| ^{\frac{p}{2}}=  \hat{Z}\left| t-t_i^n\right| ^{\frac{p}{2}} , t\in [t_i^n,t_{i+1}^n),
\end{align*}
with $\hat{Z}=(1+L_{\hat{v}})^p  Z$ and this ends the proof.
\end{proof}
The next lemma is the approximation between the implicit function $\hat{v}$ and its numerical approximation $\hat{v}^n$.
\begin{lemma}\label{v}
Let $\hat{v}$ be the implicit function from inherent regular SDE \eqref{equa7A} and $\hat{v}^n$ its numerical approximation defined in  \eqref{equa8b}, then for two time depending functions u an w defined on  $D=\cup_{q=1}^{\infty}D_q$ that agree with $u_q$ and $w_q$ on $D_q, q\geq1$, we have the following inequality
  \begin{align*}
       \|\hat{v}^n(t,	u_q(t))- \hat{v}(t,w_q(t))\|^2
       &\leq C\left(\frac{1}{n^2}+\frac{\|u_q(t)\|^2}{n^2} +\|u_q(t)-u_q(k_n(t))\|^2\right)\\
        &+C\|w_q(t)-u_q(t)\|^2,
    \end{align*} 
    C is the constant independent to $n \in \mathbb{N}$.
    \end{lemma}

   \begin{proof}
       From \eqref{equa8b} and \eqref{equa5A}, using the following composition rule
  $$a_1b_1-a_2b_2=a_1(b_1-b_2)+(a_1-a_2)b_2,$$ we have
      \begin{align*}
     	\|\hat{v}(t,w_q(t))&-\hat{v}^n(t,	u_q(t))\|^2=
        \left\|-(A(t)+R(t)B(t))^{-1}R(t)B(t)w_q(t) \right.\\
        &\left.-(A(t)+R(t)B(t))^{-1}f_{1q}(t,w_q(t))\right.\\
        &+\left. (A(k_n(t))+R(k_n(t))B(k_n(t)))^{-1}R(k_n(t))B(k_n(t))u_q(t)\right.\\
        &+\left. (A(k_n(t))+R(k_n(t))B(k_n(t)))^{-1}f_{1q}(k_n(t),u_q((k_n(t))) \right\|^2\\
        &\leq 2\left\|(A(t)+R(t)B(t))^{-1}R(t)B(t)w_q(t)\right.\\
        &\left.-(A(k_n(t))+R(k_n(t))B(k_n(t)))^{-1}R(k_n(t))B(k_n(t))u_q(t)\right\|^2\\
          &+2\left\|(A(t)+R(t)B(t))^{-1}f_{1q}(t,w_q(t))\right.\\
        &\left.- (A(k_n(t))+R(k_n(t))B(k_n(t)))^{-1}f_{1q}(k_n(t),u_q((k_n(t)))\right\|^2.
         \end{align*}
         Remember that $f_{1q}$ is the function that agrees  with $f_1=Rf$ on  $D_q, q\geq1$, $f_1$  been defined in \eqref{chapeau}.
         We therefore have 
             \begin{eqnarray}
             \label{eqm}
             \|\hat{v}(t,w_q(t))-\hat{v}^n(t,	u_q(t))\|^2
          \leq     4(I+II+III+IV),
         \end{eqnarray}
     where
     \begin{eqnarray*}
         I&:=&\left\|(A(t)+R(t)B(t))^{-1}\right\|^2\left\|R(t)B(t)w_q(t)-R(k_n(t))B(k_n(t))u_q(t) \right\|^2\\
        II &:=&\left\|(A(t)+R(t)B(t))^{-1}-(A(k_n(t))+R(k_n(t))B(k_n(t)))^{-1}\right\|^2\\
&\times&\left\|R(k_n(t))B(k_n(t))\bar{u}_q^n(t) \right\|^2\\
III&:=&\left\|(A(t)+R(t)B(t))^{-1}\right\|^2\left\|f_{1q}(t,w_q(t))-f_{1q}(k_n(t),u_q((k_n(t)))\right\|^2\\
IV&:=&\left\|(A(t)+R(t)B(t))^{-1}-(A(k_n(t))+R(k_n(t))B(k_n(t)))^{-1}\right\|^2\\
&\times&\left\|\hat{f}_q(k_n(t),u_q((k_n(t)))\right\|^2.
     \end{eqnarray*}
    Let us estimate each term. For the first term, we have the following. 
        \begin{eqnarray*}
        I&:=& \left\|(A(t)+R(t)B(t))^{-1}\right\|^2\left\|R(t)B(t)w_q(t)-R(k_n(t))B(k_n(t))u_q(t) \right\|^2\\
        &\leq& K_1^2\left\|R(t)B(t)\right\|^2\left\|w_q(t)-u_q(t) \right\|^2\\
       & +&K_1^2\left\|R(t)B(t)-R(k_n(t))B(k_n(t))\right\|^2\left\|u_q(t) \right\|^2\\
        &\leq & C(\left\|w_q(t)-u_q(t) \right\|^2+\|u_q(t) \|^2|t-k_n(t)|^2)\\
        &\leq & C(\left\|w_q(t)-u_q(t) \right\|^2+\frac{\|u_q(t) \|^2}{n^2}).
        \end{eqnarray*}
       For the second term, we also have     
\begin{eqnarray*}
II
&:=&\left\|(A(t)+R(t)B(t))^{-1}-(A(k_n(t))+R(k_n(t))B(k_n(t)))^{-1}\right\|^2\\
&\times&\left\|R(k_n(t))B(k_n(t))\bar{u}_q^n(t) \right\|^2\\
&\leq& C\|u_q(t) \|^2|t-k_n(t)|^2\\
&\leq& C\frac{\|u_q(t) \|^2}{n^2}
\end{eqnarray*}
For the estimation of the third term, we have
\begin{eqnarray*}
    III&:=&\left\|(A(t)+R(t)B(t))^{-1}\right\|^2\left\|f_{1q}(t,w_q(t))-f_{1q}(k_n(t),u_q((k_n(t)))\right\|^2\\
        &\leq&  K_1^2\left\|f_{1q}(t,w_q(t))-f_{1q}(k_n(t),u_q((k_n(t)))\right\|^2\\
&\leq& C\|w_q(t)-u_q(k_n(t)) \|^2.
\end{eqnarray*}
For the fourth term, we have the following. 
        \begin{eqnarray*}
IV&=&\left\|(A(t)+R(t)B(t))^{-1}-(A(k_n(t))+R(k_n(t))B(k_n(t)))^{-1}\right\|^2\\
&\times&\left\|\hat{f}_q(k_n(t),u_q((k_n(t)))\right\|^2\\
&\leq &L_qK_1^4|t-k_n(t)|^2\\
&\leq &\frac{C }{n^2}.
\end{eqnarray*}
 Using previous estimates in \eqref{eqm} yields the following.
         \begin{align*}
        \|\hat{v}(t,w_q(t))-\hat{v}^n(t,	u_q(t))\|^2&\leq C\|w_q(t)-u_q(k_n(t))-u_q(t)+u_q(t)\|^2\\
        & +\frac{C}{n^2}+C\|w_q(t)-u_q(t)\|^2+C\frac{\|u_q(t)\|^2}{n^2}\\
       &\leq \frac{C}{n^2}+C\frac{\|u_q(t)\|^2}{n^2} +C\|u_q(t)-u_q(k_n(t))\|^2\\
        &+C\|w_q(t)-u_q(t)\|^2.
    \end{align*} 
    This ends our proof.
     \end{proof}
     \begin{lemma}\label{v1}
Let $\hat{v}^n$ and $\hat{v}^m$ be the numerical approximations of $\hat{v}$  with $n$ and $m$ time subdivisions respectively.  For  two time depending functions u an w defined on  $D=\cup_{q=1}^{\infty}D_q$  that agree  with $u_q$ and  $w_q$ on $ D_q, q\geq1$,  we have the following inequality
  \begin{align*}
       \| \hat{v}^m(t,w_q(t))-\hat{v}^n(t,	u_q(t))\|^2
       &\leq C\left(\frac{1}{n^2} +\frac{1}{m^2}+\frac{\|u_q(t)\|^2}{n^2}+\frac{\|u_q(t)\|^2}{m^2}\right)\\
        &+C\left(\|u_q(t)-u_q(k_n(t))\|^2+\|w_q(t)-u_q(t)\|^2\right).
    \end{align*} 
    \end{lemma}
    
    \begin{proof}
    The proof follows the same lines as in Lemma \ref{v}.
    Note that  \begin{align*}
        |k_m(t)-k_n(t)|&=|k_m(t)-t+t-k_n(t)|\leq |k_m(t)-t|+|t-k_n(t)|\\
        &\leq \frac{T}{n}+\frac{T}{m}.
              \end{align*}   
    \end{proof}

    \begin{lemma} \label{exis}
     Assume that Assumptions \ref{assum1} and \ref{assum2} hold, then the numerical scheme \eqref{equa16b}  defined on $D_q,~q\geq1$   converges in probability,  uniformly in $[0, T ]$ to  the unique function $U_q$, solution of the SDE \eqref{equa7A}  defined on $D_q,~q\geq1$ as $n\to \infty$. 
     \end{lemma}

     \begin{proof}

	We have assumed that the functions   $\hat{f}(\cdot)$ and $\hat{g}(\cdot)$ are locally Lipschitz in D. Then by remark \ref{def4}, for every $q>0$ we have a bounded measurable function $\hat{f}_q(\cdot, \cdot)$ and $\hat{g}_q(\cdot,\cdot)$ defined on: $\mathbb{	R}_+\times \mathbb{	R}^d\to \mathbb{	R}^d$ and $\mathbb{	R}_+\times \mathbb{	R}^d\to \mathbb{	R}^{d\times d_1}$ respectively, such that $\hat{f}_q(\cdot,\cdot)$ and $\hat{f}(\cdot,\cdot)$
    agree on $\left[0,T \right] \times D_q$  and also $g_q(\cdot,\cdot)$ and $g(\cdot,\cdot)$. In addition, we also have
	\begin{align*}
		\left\|\hat{f}_q(t,u_1)-\hat{f}_q(t,u_2) \right\|& \leq L_q\left\| u_1-u_2\right\|; t\in \left[0,T \right]~, ~u_1,u_2\in \mathbb{R}^d,~~~~~~~~~~~~~~~~~~~~~~~~~~~~~~~~\\
		\left|\hat{g}_q(t,u_1)-\hat{g}_q(t,u_2) \right|_F& \leq L_q\left\| u_1-u_2\right\|; t\in \left[0,T \right]~, ~u_1,u_2\in \mathbb{R}^d.
	\end{align*}
	Let $\bar{u}^n_q(\cdot), ~n\in \mathbb{	N}, ~~q>0$ denotes the Euler's approximation defined by the equation \eqref{equa17b} where the functions $f(\cdot, \cdot)$ and $g(\cdot, \cdot)$ are replaced  by the functions $f_q(\cdot, \cdot)$ and $g_q(\cdot, \cdot)$ respectively.
	We also define the stopping time
	$$\tau_n^q=T\land \inf \left\lbrace t\geq 0,~ u^q_n(t)\notin D_q \right\rbrace .~~~~~~~$$ 
	Note that the approximations $\bar{u}^n(\cdot)$ and $\bar{u}^n_q(\cdot)$ agree on $ \left( 0,\tau_n^q\right],~ n\in \mathbb{	N}, ~~q>0$. 
	From \eqref{equa17b}, we also have 
	\begin{eqnarray}
	\label{suit1}
		\bar{u}^n_q(t) &=&\bar{u}^n_q(0))+\int_{0}^{t} P'(k_n(s))\left[ 	\bar{u}_q^n(s)+\hat{v}^n(s,	\bar{u}_q^n(s))\right]\nonumber \\
        &+&\int_{0}^{t}A^-(k_n(s)) B(k_n(s))
		\left[	\bar{u}_q^n(s)+\hat{v}^n(s,	\bar{u}_q^n(s))\right]ds \nonumber\\
		&+& \int_{0}^{t}\hat{f}_q(k_n(s),\bar{u}^n_q(k_n(s)))ds+\int_{0}^{t}\hat{g}_q(k_n(s),\bar{u}^n_q(k_n(s)))dW(s),
	\end{eqnarray}

	and 
	\begin{eqnarray}\label{suit2}
		\bar{u}_q^m(t) &=&\bar{u}_q^m(0))+\int_{0}^{t} P'(k_m(s))\left[ \bar{u}_q^m(s)+\hat{v}^m(s,\bar{u}_q^m(s))\right]\nonumber\\
        &+&\int_{0}^{t} A^-(k_m(s)) B(k_m(s))\left[\bar{u}_q^m(s)+\hat{v}^m(s,\bar{u}_q^m(s))\right]ds \nonumber\\
		&+&\int_{0}^{t}\hat{f}_q(k_m(s),\bar{u}_q^m(k_m(s)))ds+\int_{0}^{t}\hat{g}_q(k_m(s),\bar{u}_q^m(k_m(s)))dW(s).\nonumber\\
	\end{eqnarray}
	Remember that  $k_n(t)=t_i^n=\frac{i T}{n} ,~ t\in \left[ t_i^n ,t_{i+1}^n\right) $.
	The difference between \eqref{suit1} and \eqref{suit2} yields
	\begin{eqnarray}
			\bar{u}^n_q(t)-	\bar{u}_q^m(t) 
			&=& \int_{0}^{t} P'(k_n(s))\left[ 	\bar{u}_q^n(s)+\hat{v}^n(s,	\bar{u}_q^n(s))\right] ds\nonumber\\
            &-& \int_{0}^{t}P'(k_m(s))\left[ \bar{u}_q^m(s)+\hat{v}^m(s,	\bar{u}_q^m(s))\right]ds\nonumber\\
		& +& \int_{0}^{t}A^-(k_n(s)) B(k_n(s))\left[	\bar{u}_q^n(s)+\hat{v}^n(s,	\bar{u}_q^n(s))\right]\nonumber\\
        &-&A^-(k_m(s)) B(k_m(s))\left[\bar{u}_q^m(s)+\hat{v}^m(s,\bar{u}_q^m(s))\right]ds\nonumber \\
		&+&\int_{0}^{t}\hat{f}_q(k_n(s),\bar{u}^n_q(k_n(s)))-\hat{f}_q(k_m(s),\bar{u}^m_q(k_m(s)))ds\nonumber\\ 
		&+&\int_{0}^{t}\hat{g}_q(k_n(s),\bar{u}^n_q(k_n(s)))-\hat{g}_q(k_m(s),\bar{u}^m_q(k_m(s)))dW(s).\nonumber\\
	\end{eqnarray}
	We apply the Ito's  lemma to $\left\| 	\bar{u}^n_{q}(t)-\bar{u}^m_{q}(t)\right\|^2,~ t\in \left[0,T \right] $ and  obtain
	\begin{align}\label{equa9}
			&\left\| 	\bar{u}^n_{q}(t)-\bar{u}^m_{q}(t)\right\|^2 
		=2	\int_{0}^{t}\left\langle 	\bar{u}^n_q(s)-	\bar{u}_q^m(s), P'(k_n(s))\left[ 	\bar{u}_q^n(s)+\hat{v}^n(s,	\bar{u}_q^n(s))\right] \right\rangle ds\nonumber\\
		-&2	\int_{0}^{t}\left\langle 	\bar{u}^n_q(s)-	\bar{u}_q^m(s), P'(k_m(s))\left[ 	\bar{u}_q^m(s)+\hat{v}^m(s,	\bar{u}_q^m(s))\right]\right\rangle ds\nonumber\\
		+&2 \int_{0}^{t}\left\langle 	\bar{u}^n_q(s)-	\bar{u}_q^m(s),A^-(k_n(s)) B(k_n(s))\left[	\bar{u}_q^n(s)+\hat{v}^n(s,	\bar{u}_q^n(s))\right] \right \rangle ds\nonumber\\
		-&2 \int_{0}^{t}\left\langle 	\bar{u}^n_q(s)-	\bar{u}_q^m(s), A^-(k_m(s)) B(k_m(s))\left[\bar{u}_q^m(s)+\hat{v}^m(s,	\bar{u}_q^m(s))\right]\right\rangle ds\nonumber\\
		+&2\int_{0}^{t}\left\langle 	\bar{u}^n_q(s)-	\bar{u}_q^m(s),\hat{f}_q(k_n(s),\bar{u}^n_q(k_n(s)))-\hat{f}_q(k_m(s),\bar{u}^m_q(k_m(s))) \right\rangle ds\nonumber\\
		+&\int_{0}^{t} \left|\hat{g}_q(k_n(s),\bar{u}^n_q(k_n(s)))
		-\hat{g}_q(k_m(s),\bar{u}^m_q(k_m(s)))\right|_F^2d(s), \nonumber\\
		+&2\int_{0}^{t} \left\langle 	\bar{u}^n_q(s)-	\bar{u}_q^m(s), \hat{g}_q(k_n(s),\bar{u}^n_q(k_n(s))) - \hat{g}_q(k_m(s),\bar{u}^m_q(k_m(s)))\right\rangle dW(s).
	\end{align}
	To ease the presentation, we define $a(s)$, $b(s)$, $c(s)$, $d(s)$ and $e(s)$ as
	\begin{eqnarray*}
	a(s)&=&2\left\langle \bar{u}^n_q(s)-	\bar{u}_q^m(s), P'(k_n(s))\left[ 	\bar{u}_q^n(s)+\hat{v}^n(s,	\bar{u}_q^n(s))\right] \right\rangle\\
	 &-& 2\left\langle \bar{u}^n_q(s)-	\bar{u}_q^m(s), P'(k_m(s))\left[ 	\bar{u}_q^m(s)+\hat{v}^m(s,	\bar{u}_q^m(s))\right]\right\rangle,~~~~~~~~~~~~~~~~~~~~~~~~
	\end{eqnarray*}
	\begin{eqnarray*}
			b(s)&=&2	\left\langle 	\bar{u}^n_q(s)-	\bar{u}_q^m(s),A^-(k_n(s)) B(k_n(s))\left[	\bar{u}_q^n(s)+\hat{v}^n(s,	\bar{u}_q^n(s))\right] \right\rangle \\
		&-& 2\left\langle 	\bar{u}^n_q(s)-	\bar{u}_q^m(s), A^-(k_m(s)) B(k_m(s))\left[	\bar{u}_q^m(s)+\hat{v}^m(s,	\bar{u}_q^m(s))\right]\right\rangle,~~~~~~~~~~~~
	\end{eqnarray*}
	\begin{eqnarray*}
		c(s)&=&2\left\langle 	\bar{u}^n_q(s)-	\bar{u}_q^m(s),\hat{f}_q(k_n(s),\bar{u}^n_q(k_n(s)))-\hat{f}_q(k_m(s),\bar{u}^m_q(k_m(s)))\right\rangle,
	\end{eqnarray*}
	\begin{eqnarray*}
		d(s)&=&2\left\langle 	\bar{u}^n_q(s)-	\bar{u}_q^m(s), \hat{g}_q(k_n(s),\bar{u}^n_q(k_n(s)))-\hat{g}_q(k_m(s),\bar{u}^m_q(k_m(s)))\right\rangle,      
	\end{eqnarray*}
	\begin{eqnarray*}
		e(s)=\left|\hat{g}_q(k_n(s),\bar{u}^n_q(k_n(s)))-\hat{g}_q(k_m(s),\bar{u}^m_q(k_m(s)))\right|_F^2.
	\end{eqnarray*}
	Then for all $p\geq 2$ \eqref{equa9} becomes
	\begin{align*}
		\left\| 	\bar{u}^n_{q}(t)-\bar{u}^m_{q}(t)\right\|^{2p}&\leq 5^{p-1} \left( \left| \int_{0}^{t}a(s) ds\right|^p+\left| \int_{0}^{t}b(s)d(s)\right|^p + \left| \int_{0}^{t}c(s) d(s)\right|^p\right.\\
		&\left.+\left| \int_{0}^{t}d(s)dW(s)\right|^p+\left| \int_{0}^{t}e(s)ds\right|^p\right) \\
		& \leq 5^{p-1} \left( \left\lbrace \left| \int_{0}^{t}a(s) ds\right|^2\right\rbrace ^{\frac{p}{2}}+\left\lbrace \left| \int_{0}^{t}b(s)d(s)\right|^2\right\rbrace^ {\frac{p}{2}} \right.\\
		&	\left.+\left\lbrace \left| \int_{0}^{t}c(s)d(s)\right|^2\right\rbrace^ {\frac{p}{2}}+ \left\lbrace \left| \int_{0}^{t}e(s)d(s)\right|^2\right\rbrace^ {\frac{p}{2}}\right.\\
        &\left.+ \left| \int_{0}^{t}d(s) dW(s)\right|^p\right),t\in \left[0,T \right] .
	\end{align*}
	Using Cauchy Schwartz inequality, we obtain
	\begin{align*}
		\left\|	\bar{u}^n_{q}(t)-\bar{u}^m_{q}(t)\right\| ^{2p}& \leq 5^{p-1} \left(T^{\frac{p}{2}} \left\lbrace \int_{0}^{t}\left| a(s)\right|^2 ds\right\rbrace ^{\frac{p}{2}}+T^{\frac{p}{2}}\left\lbrace  \int_{0}^{t}\left|b(s)\right|^2d(s)\right\rbrace^ {\frac{p}{2}} \right.\nonumber\\
		&\left.+T^{\frac{p}{2}} \left\lbrace \int_{0}^{t}\left| c(s)\right|^2 ds\right\rbrace ^{\frac{p}{2}}+	T^{\frac{p}{2}} \left\lbrace \int_{0}^{t}\left| e(s)\right|^2 ds\right\rbrace ^{\frac{p}{2}}\right.\\
        &\left.+ \left| \int_{0}^{t}d(s) dW(s)\right|^p\right),~t\in \left[0,T \right] .
	\end{align*}
	Multiplying by the $\exp(-p\left\|\zeta \right\| )$ \footnote{Remember that $\zeta$ is the initial solution of \eqref{equa1a}} yields
	\begin{align*}
		\sup_{s\leq t}\left( \exp(-p\left\|\zeta \right\| )	\left\|	\bar{u}^n_{q}(s)-\bar{u}^m_{q}(s)\right\| ^{2p}\right) & \leq 5^{p-1}T^{\frac{p}{2}} \left\lbrace \int_{0}^{t} \exp(-2\left\|\zeta \right\| )\left| a(s)\right|^2 ds\right\rbrace ^{\frac{p}{2}}\nonumber\\
		&+5^{p-1}T^{\frac{p}{2}}\left\lbrace  \int_{0}^{t}\exp(-2\left\|\zeta \right\| )\left|b(s)\right|^2d(s)\right\rbrace^ {\frac{p}{2}}\nonumber\\
		& + 5^{p-1}T^{\frac{p}{2}} \left\lbrace \int_{0}^{t} \exp(-2\left\|\zeta \right\| )\left| c(s)\right|^2 ds\right\rbrace ^{\frac{p}{2}}\nonumber\\
		& + 5^{p-1}T^{\frac{p}{2}} \left\lbrace \int_{0}^{t} \exp(-2\left\|\zeta \right\| )\left| e(s)\right|^2 ds\right\rbrace ^{\frac{p}{2}}\nonumber\\
		& + 5^{p-1}\exp(-p\left\|\zeta \right\| )\left| \int_{0}^{t}d(s) dW(s)\right|^p.
	\end{align*}
	Taking the expectation and using the Burkholder inequality yields the following.
	\begin{align}\label{equa10}
		\mathbb{	E}	\sup_{s\leq t}\left( \exp(-p\left\|\zeta \right\| )	\left\|	\bar{u}^n_{q}(s)-\bar{u}^m_{q}(s)\right\| ^{2p}\right) & \leq 5^{p-1}T^{\frac{p}{2}} \mathbb{	E}\left\lbrace \int_{0}^{t} \exp(-2\left\|\zeta \right\| )\left| a(s)\right|^2 ds\right\rbrace ^{\frac{p}{2}}\nonumber\\
		&+5^{p-1}T^{\frac{p}{2}}\mathbb{	E}\left\lbrace  \int_{0}^{t}\exp(-2\left\|\zeta \right\| )\left|b(s)\right|^2d(s)\right\rbrace^ {\frac{p}{2}}\nonumber\\
		&+5^{p-1}T^{\frac{p}{2}}\mathbb{	E}\left\lbrace  \int_{0}^{t}\exp(-2\left\|\zeta \right\| )\left|c(s)\right|^2d(s)\right\rbrace^ {\frac{p}{2}}\nonumber\\
		&+5^{p-1}T^{\frac{p}{2}}\mathbb{	E}\left\lbrace  \int_{0}^{t}\exp(-2\left\|\zeta \right\| )\left|e(s)\right|^2d(s)\right\rbrace^ {\frac{p}{2}}\nonumber\\
		& +5^{p-1}C_p \mathbb{	E}\left\lbrace \int_{0}^{t}\exp(-2\left\|\zeta \right\| )\left| d(s)\right|^2 ds\right\rbrace ^{\frac{p}{2}},
	\end{align}	
    where $C_p=\left(\frac{p^{p+1}}{2(p-1)^{p-1}} \right)^{\frac{p}{2}} $.
    
	Let us evaluate $\left| a(s)\right|^2, \left| b(s)\right|^2 \text{and }~\left| c(s)\right|^2, s\in \left[0,T \right].$
	We have
	\begin{align*}
		\left| a(s)\right|^2&=\left\langle \bar{u}^n_q(s)-	\bar{u}_q^m(s), P'(k_n(s))\left[ 	\bar{u}_q^n(s)+\hat{v}^n(s,	\bar{u}_q^n(s))\right] \right\rangle\nonumber\\
        &-\left\langle \bar{u}^m_q(s)-	\bar{u}_q^m(s), P'(k_m(s))\left[ 	\bar{u}_q^m(s)+\hat{v}^m(s,	\bar{u}_q^m(s))\right] \right\rangle\nonumber\\
		&\leq\left( \left\| \bar{u}^n_q(s)-	\bar{u}_q^m(s)\right\| ^2\right.\nonumber\\
        &\left.+\left\|  P'(k_n(s))\left[ 	\bar{u}_q^n(s)+\hat{v}^n(s,	\bar{u}_q^n(s))\right]
	       - P'(k_m(s))\left[ 	\bar{u}_q^m(s)+\hat{v}^m(s,	\bar{u}_q^m(s))\right]\right\| ^2\right) ^2\\ 
	     &\leq\left( \left\| \bar{u}^n_q(s)-	\bar{u}_q^m(s)\right\| ^2\right.\nonumber\\
         &\left.+2\left|  P'(k_n(s)\right|^2 _F\left\|  	\bar{u}_q^n(s)+\hat{v}^n(s,	\bar{u}_q^n(s))-\bar{u}_q^m(s)-\hat{v}^m(s,	\bar{u}_q^m(s))\right\| ^2\right.\nonumber\\
		&\left. +2\left\|\bar{u}_q^m(s)+\hat{v}^m(s,	\bar{u}_q^m(s)) \right\|^2 \left[ \left|  P'(k_n(s))-  P'(k_m(s))\right| _F^2	\right]\right) ^{\frac{2p}{p}}.
        \end{align*}
We use the fact that the function $f_q(\cdot)$ and the matrix $P'(\cdot)$ are bounded and  $f_q(\cdot)$ satisfies the Lipschtiz condition.
        \begin{align}\label{equa11aa}
		\left| a(s)\right|^2 
		&\leq\left( \left\| \bar{u}^n_q(s)-	\bar{u}_q^m(s)\right\| ^2\right.\nonumber\\
        &\left.+4\left|  P'\right| _F^2\left( \left\|  	\bar{u}_q^n(s)-\bar{u}_q^m(s)\right\| ^2+\left\| \hat{v}^n(s,	\bar{u}_q^n(s))-\hat{v}^m(s,	\bar{u}_q^m(s))\right\| ^2\right) \right.\nonumber\\
		&\left. +2\left\|\bar{u}_q^m(s)+\hat{v}^m(s,	0)+\hat{v}^m(s,	\bar{u}_q^m(s))-\hat{v}^m(s,	0) \right\|^2 \left[ \left|  P'(k_n(s))-  P'(k_m(s))\right| _F^2	\right]\right) ^{\frac{2p}{p}}\nonumber\\
		&\leq\left( \left\| \bar{u}^n_q(s)-	\bar{u}_q^m(s)\right\| ^2+4K_1(\left\|  	\bar{u}_q^n(s)-\bar{u}_q^m(s)\right\| ^2 +\left\| \hat{v}^n(s,	\bar{u}_q^n(s))-\hat{v}^m(s,	\bar{u}_q^m(s))\right\| ^2)\right.\nonumber\\
		&\left. +4K_1\left( \left\|\bar{u}_q^m(s)+\hat{v}^m(s,	0)\right\| ^2+\left\| \hat{v}^m(s,	\bar{u}_q^m(s))-\hat{v}^m(s,	0) \right\|^2 \right) \left\| k_n(s)-  k_m(s)\right\|^2\right) ^{\frac{2p}{p}}\nonumber\\
		&\leq3^{p-1}\left(\left[ 1+4K_1\right]^{p} \left\|  	\bar{u}_q^n(s)-\bar{u}_q^m(s)\right\| ^{2p}+4K_1\left\| \hat{v}^n(s,	\bar{u}_q^n(s))-\hat{v}^m(s,	\bar{u}_q^m(s))\right\| ^{2p}\right.\nonumber\\ 
		& \left.+\left( 4K_1\left( (2+L_{\hat{v}})\left\|\bar{u}_q^m(s)\right\|^2 +2\left\| \hat{v}^m(s,	0)\right\| ^2 \right)\right) ^{p} \left\| k_n(s)-  k_m(s)\right\|^{2p}\right) ^{\frac{2}{p}}.		
	\end{align}
	We  also have  $ \left\| k_n(s)-  k_m(s)\right\|\leq \left\| k_n(s)-s\right\|+\left\| s-k_m(s)\right\|\leq \frac{T}{n}+\frac{T}{m}$ and \eqref{equa11aa} becomes
	\begin{align}\label{equa11a}
		\left| a(s)\right|^2&\leq3^{p-1}\left(\left[ 1+4K_1\right]^{p} \left\|  	\bar{u}_q^n(s)-\bar{u}_q^m(s)\right\| ^{2p}+4K_1\left\| \hat{v}^n(s,	\bar{u}_q^n(s))-\hat{v}^m(s,	\bar{u}_q^m(s))\right\| ^{2p}\right.\nonumber\\
		& \left.+3^{2p-1}\left( 4K_1\left( (2+L_{\hat{v}})\left\|\bar{u}_q^m(s)\right\|^2 +2\left\| \hat{v}^m(s,	0)\right\| ^2 \right)\right) ^{2p}\left[  \left(  \frac{T}{n}\right) ^{2p}+ \left( \frac{T}{m}\right) ^{2p}\right] \right) ^{\frac{2}{p}}	.
	\end{align}
  
	Additionally, we use the same method as in the estimation of $a(\cdot)$, and we have for $s\in \left[0,T \right]$
	\begin{align}\label{equa12a}
		\left| b(s)\right|^2&\leq3^{p-1}\left(\left[ 1+4K_1)\right]^{p} \left\|  	\bar{u}_q^n(s)-\bar{u}_q^m(s)\right\| ^{2p}+4K_1\left\| \hat{v}^n(s,	\bar{u}_q^n(s))-\hat{v}^m(s,	\bar{u}_q^m(s))\right\| ^{2p}\right.\nonumber\\
		& \left.+3^{2p-1}\left( 4K_1\left( (2+L_{\hat{v}})\left\|\bar{u}_q^m(s)\right\|^2 +2\left\| \hat{v}^m(s,	0)\right\| ^2 \right)\right) ^{2p}\left[  \left(  \frac{T}{n}\right) ^{2p}+ \left( \frac{T}{m}\right) ^{2p}\right] \right) ^{\frac{2}{p}}	.
	\end{align}
   
	In this case, we use the  fact that
	\begin{eqnarray*}
	A(k_n(s))B(k_n(s))-A(k_m(s))B(k_m(s)&=&A(k_n(s))(B(k_n(s))-B(k_m(s)))\\
	&+&(A(k_n(s))-A(k_m(s)))B(k_m(s)).
	\end{eqnarray*}
	
	Moreover we can also estimated $c(\cdot)$ as  follows
	\begin{align}\label{equa13a}
		\left| c(s)\right|^2&=\left( 2\left\langle 	\bar{u}^n_q(s)-	\bar{u}_q^n(s),\hat{f}_q(k_n(s),\bar{u}^n_q(k_n(s)))-\hat{f}_q(k_m(s),\bar{u}^m_q(k_m(s)))\right\rangle \right) ^2\nonumber\\
		&\leq\left( \left\| \bar{u}^n_q(s)-	\bar{u}_q^m(s)\right\| ^2+\left\| \hat{f}_q(\bar{u}^n_q(k_n(s)))-\hat{f}_q(\bar{u}^m_q(k_m(s)))\right\|^2  \right)^{p\frac{2}{p}}\nonumber\\
		&\leq\left( \left\| \bar{u}^n_q(s)-	\bar{u}_q^m(s)\right\| ^2+K_1L_q(1+L_{\hat{v}})\left\| \bar{u}^n_q(k_n(s))-	\bar{u}^m_q(k_m(s))\right\| ^2  \right)^{p\frac{2}{p}},	~s\in \left[0,T \right].	
	\end{align}
	But we have
	\begin{align*}
		\left\| \bar{u}^n_q(k_n(s))-\bar{u}^m_q(k_m(s))\right\|&=\left\| \bar{u}^n_q(k_n(s))-\bar{u}^n_q(s)+\bar{u}^n_q(s)-\bar{u}^m_q(s)+\bar{u}^m_q(s)-\bar{u}^m_q(k_m(s))\right\|\\
		&\leq\left\| \bar{u}^n_q(k_n(s))-\bar{u}^n_q(s)\right\| +\left\| \bar{u}^m_q(s)-\bar{u}^m_q(k_m(s))\right\|\\
        &+\left\| \bar{u}^n_q(s)-\bar{u}^m_q(s)\right\|.
	\end{align*}
	Then \eqref{equa13a} becomes
	\begin{align}\label{equa11}
		\left| c(s)\right|^2	&\leq\left( C(K_1,L_q,L_{\hat{v}})\left[\left\| \bar{u}^n_q(s)-	\bar{u}_q^m(s)\right\| ^2\right.\right.\nonumber\\
        & \left.\left.+ \left\| \bar{u}^n_q(k_n(s))-\bar{u}^n_q(s)\right\|^2 +\left\| \bar{u}^m_q(s)-\bar{u}^m_q(k_m(s))\right\|^2\right]\right)^{p\frac{2}{p}}\nonumber\\
		&\leq\left( C(K_1,L_q,L_{\hat{v}},p)\left[\left\| \bar{u}^n_q(s)-	\bar{u}_q^m(s)\right\| ^{2p}\right.\right.\nonumber\\
        &\left.\left.+\left\| \bar{u}^n_q(k_n(s))-\bar{u}^n_q(s)\right\|^{2p} +\left\| \bar{u}^m_q(s)-\bar{u}^m_q(k_m(s))\right\|^{2p}\right]  \right)^{\frac{2}{p}}\nonumber\\
	\end{align}
	Additionally, we use the same method as in the estimation of $c(\cdot)$ to have the estimation of $d(\cdot)$.  For $s\in \left[0,T \right]$, we have
	\begin{align}\label{equa11b}
		\left| d(s)\right|^2 &\leq\left( C(K_1,L_q,L_{\hat{v}},p) \left[\left\| \bar{u}^n_q(s)-	\bar{u}_q^m(s)\right\| ^{2p}\right.\right.\nonumber\\
        &\left.\left.+\left\| \bar{u}^n_q(k_n(s))-\bar{u}^n_q(s)\right\|^{2p} +\left\| \bar{u}^m_q(s)-\bar{u}^m_q(k_m(s))\right\|^{2p}\right]  \right)^{\frac{2}{p}}.
	\end{align}
	Finally, let us estimate $e(\cdot)$.
	\begin{align}\label{equa12}
		&\left| e(s)\right| ^2
		=\left(\left|\hat{g}_q(k_n(s),\bar{u}^n_q(k_n(s)))-\hat{g}_q(k_m(s),\bar{u}^m_q(k_m(s))\right|_F^2\right) ^2\nonumber\\
                &\leq C(K_1,L_q,L_{\hat{v}},p)\left(\left\| \bar{u}^n_q(k_n(s))-\bar{u}^m_q(k_m(s))\right\| ^{2p}\right) ^{\frac{2}{p}}\nonumber\\
		&\leq C(K_1,L_q,L_{\hat{v}},p) \left[ \left\| \bar{u}^n_q(s)-	\bar{u}_q^m(s)\right\| ^{2p}\right.\nonumber\\
        &\left.+\left\| \bar{u}^n_q(k_n(s))-\bar{u}^n_q(s)\right\|^{2p} +\left\| \bar{u}^m_q(s)-\bar{u}^m_q(k_m(s))\right\|^{2p} \right]^{\frac{2}{p}}.				
	\end{align}
	Replacing \eqref{equa11a}, \eqref{equa12a} , \eqref{equa11}, \eqref{equa11b} and    \eqref{equa12} in  \eqref{equa10} and using Cauchy inequality  yields
	\begin{align}\label{equa14a2}
		\mathbb{	E}	\sup_{s\leq t}&\left( \exp(-p\left\|\zeta \right\| )	\left\| \bar{u}_q^{n}(s)-	\bar{u}_q^{m}(s)\right\| ^{2p}\right)  \leq C  \mathbb{	E}\int_{0}^{t} \exp(-p\left\|\zeta \right\| )\left\|  \bar{u}_q^{n}(s)-	\bar{u}_q^{m}(s)\right\| ^{2p} ds \nonumber\\
		&+\left( \mathbb{	E}\int_{0}^{t}  \exp(-p\left\|\zeta \right\| )\left\| \bar{u}_q^{m}(s)\right\| ^{2p}ds\right)\left( \frac{C}{m^{2p}} +\frac{C}{n^{2p}}\right) \nonumber\\
        & +C  \mathbb{	E}\int_{0}^{t} \exp(-p\left\|\zeta \right\| )\left\| \bar{u}^m_q(s)-\bar{u}^n_q(k_m(s))\right\|^{2p} ds \nonumber\\
          &+ C  \mathbb{	E}\int_{0}^{t} \exp(-p\left\|\zeta \right\| )\left\| \bar{u}^n_q(s)-\bar{u}^n_q(k_n(s))\right\|^{2p} ds \nonumber\\
		&+\frac{C}{m^{2p}} +\frac{C}{n^{2p}} +C  \mathbb{	E}\int_{0}^{t} \exp(-p\left\|\zeta \right\| )\left\|\hat{v}^n(s,\bar{u}_q^n(s)-\hat{v}^m(s,\bar{u}^m_q(s))\right\|^{2p} ds.
        \end{align}

By using  \eqref{equa41a}  in the proof of Theorem \ref{theo3a}, for $s\in \left[0,T \right]$, we have 
	\begin{align}\label{WES}
		\mathbb{	E}\left\| \bar{u}^m_q(s)-\bar{u}^m_q(k_m(s))\right\|^{2p}\leq  Z \left| k_m(s)-s\right|^p\leq \frac{C}{m}.	
        \end{align}
  From Lemma \ref{v1} we have 
     \begin{align*}
	\mathbb{E}\sup_{t\leq T}\exp{(-p\|\zeta\|) }&\left\|\hat{v}^n(t, \bar{u}_q^{n}(t))-\hat{v}^m(t,  \bar{u}_q^{m}(t))\right\|^{2p} \leq  \frac{C}{n^{2p}}+\frac{C}{m^{2p}}\\
    &+C\mathbb{E}\left(\sup_{t\leq T}\exp{(-p\|\zeta\|) }\|\bar{u}_q^n(t)-\bar{u}_q^n(k_n(t))\|^{2p}\right)\\
        &+C\mathbb{E}\left(\sup_{t\leq T}\exp{(-p\|\zeta\|) }\|\bar{u}_q^m(t)-\bar{u}_q^n(t)\|^{2p}\right)\\
        &+C\frac{\mathbb{E}\left(\sup_{t\leq T}\exp{(-p\|\zeta\|) }\|\bar{u}_q^n(t)\|^2\right)}{n^{2p}}.
	\end{align*}
    From Lemma \ref{theo3a} and Lemma \ref{lemef} we have
     \begin{align}\label{suit15}
	\mathbb{E}\sup_{t\leq T}\exp{(-p\|\zeta\|) }&\left\|\hat{v}^n(t, \bar{u}_q^{n}(t))-\hat{v}^m(t, \bar{u}_q^m(t))\right\|^{2p} \leq  \frac{C}{n^{p}}+\frac{C}{m^{p}}\nonumber\\
        &+C\mathbb{E}\left(\sup_{t\leq T}\exp{(-p\|\zeta\|) }\|\bar{u}_q^m(t)-\bar{u}_q^n(t)\|^{2p}\right).
	\end{align}
        We insert \eqref{WES} and \eqref{suit15} in  \eqref{equa14a2}  then we have
        \begin{align}\label{equa14}
		\mathbb{	E}	&\sup_{s\leq t}\left( \exp(-p\left\|\zeta \right\| )	\left\| \bar{u}_q^{n}(s)-	\bar{u}_q^{m}(s)\right\| ^{2p}\right) 
		 \leq C\int_{0}^{t}  \mathbb{	E}\sup_{t\leq T} \exp(-p\left\|\zeta \right\| )\left\|    \bar{u}_q^{n}(s)-	\bar{u}_q^{m}(s)\right\| ^{2p} ds \nonumber\\
		&+\left( C\int_{0}^{t}  \exp(-p\left\|\zeta \right\| )\mathbb{	E}\left\|  \bar{u}_q^{m}(s)\right\|^{2p}ds\right) \left( \frac{C}{m^{p}} +\frac{C}{n^{p}}\right)+\frac{C}{m^{p}} +\frac{C}{n^{p}},
	\end{align}
	where $C=C(p,q, T, Z, K_1,L_{\hat{v}})$.\\
 Using Lemma \ref{lemef} we have
	\begin{align}\label{equa41b}
		\mathbb{	E}	\sup_{s\leq t}\left( \exp(-p\left\|\zeta \right\| )	\left\| \bar{u}_q^{n}(s)-	\bar{u}_q^{m}(s)\right\| ^{2p}\right) 
		& \leq C\int_{0}^{t}  \mathbb{	E}\sup_{s\leq t} \exp(-p\left\|\zeta \right\| )\left\|    \bar{u}_q^{n}(s)-	\bar{u}_q^{m}(s)\right\| ^{2p} ds \nonumber\\
		&+\frac{C}{m^{p}} +\frac{C}{n^{p}},~t\in \left[0,T \right].
	\end{align}
	Using Gronwall's lemma yields
	\begin{align}\label{equa17w}
		\mathbb{	E}	\sup_{s\leq t}\left( \exp(-p\left\|\zeta \right\| )	\left\| \bar{u}_q^{n}(s)-	\bar{u}_q^{m}(s)\right\| ^{2p}\right) 
		&\leq C\left( \frac{1}{n^p}+\frac{1}{m^p}\right),t\in \left[0,T \right] . ~~~~~~~~~~~~~~~~~~~~~~~~~~~~~~~~~~
	\end{align}
	Then the $(\bar{u}^n_q(t))_n\geq 1$, which is a Cauchy sequence converges  in probability to the continuous stochastic process that we denote $U_q(t)$, uniformly in $t\in \left[0,T \right] $.
Furthermore, we can prove  the following uniform convergence in probability for  $t\in\left[0,T \right] $ and  $n\to \infty$
	\begin{eqnarray}
    \label{neqf}
    \int_{0}^{t} P'(k_n(s))\left[ 	\bar{u}_q^n(s)+\hat{v}^n(s,	\bar{u}_q^n(s))\right]ds\to \int_{0}^{t} P'(s)\left[ 	U_q(s)+\hat{v}(s,U_q(s))\right] ds.
    \end{eqnarray}
    $$\int_{0}^{t} A^-(k_n(s))B(k_n(s))\left[ 	\bar{u}_q^n(s)+\hat{v}^n(s,	\bar{u}_q^n(s))\right]ds\to \int_{0}^{t} A^-B(s)\left[ 	U_q(s)+\hat{v}(s,U_q(s))\right]ds,$$ Indeed, \eqref{neqf} uses the fact that
	\begin{eqnarray*}
		&\mathbb{	E}\left\| \int_{0}^{t} P'(k_n(s))\left[ 	\bar{u}_q^n(s)+\hat{v}^n(s,	\bar{u}_q^n(s))\right]-P'\left[ 	U_q(s)+\hat{v}(s,	U_q(s))\right]ds\right\| ^2 \\
        & \leq C\mathbb{E}\sup_{s\leq t}\left\| \bar{u}^n_q(s)-	U_q(s)\right\| ^2+\frac{C}{n}.
	\end{eqnarray*}
	Moreover, we can prove  the following uniform convergence in probability for  $t\in\left[0,T \right] $ and  $n\to \infty$
	\begin{eqnarray}
	\label{eqq}
	&\int_{0}^{t}\hat{f}_q(k_n(s),\bar{u}^n_q(k_n(s)))ds\to\int_{0}^{t}\hat{f}_q(s,U_q(s)))ds
	\end{eqnarray}
    \begin{eqnarray*}
	\int_{0}^{t}\hat{g}_q(k_n(s),\bar{u}^n_q(k_n(s)))dW(s)\to \int_{0}^{t}\hat{g}_q(s,U_q(s))dW(s) 
    \end{eqnarray*}
	  Indeed \eqref{eqq} uses the fact that
	\begin{eqnarray*}
		&&\mathbb{E}\left\| \int_{0}^{t}\hat{f}_q(k_n(s),\bar{u}^n_q(k_n(s)))  -\hat{f}_q(s,U_q(s)))ds\right\| ^2\\
        &\leq&
		 C\,\mathbb{E}\int_{0}^{t}\left\|\hat{f}_q(k_n(s),\bar{u}^n_q(k_n(s)))-\hat{f}_q(s,U_q(s))\right\| ^2ds\\
        &\leq&
		 C\,\mathbb{E}\sup_{s\leq t}\left\| \bar{u}^n_q(s)-	U_q(s)\right\| ^2.
	\end{eqnarray*} 
    Therefore, $U_q(t)$ is the unique solution of the SDE \eqref{equa7A}. This ends the proof.
     \end{proof}

At this point, we have enough key results for the proof of Theorem \ref{theo3}.
\subsection{Proof of  Theorem \ref{theo3}}
	From that fact that   $$\bar{X}^n(t)=\bar{u}^{n}(t) +\hat{v}^n(t,\bar{u}^{n}(t)) \text{ and } Y(t)=U(t)+\hat{v}(t,U(t)), $$ 
    we have
   \begin{align*}
		\left\| 	\bar{X}^{n}(t)-	Y(t)\right\|^{2p}&=\left\| 	\bar{u}^{n}(t)+\hat{v}^n(t, \bar{u}^{n}(t))-		U(t)-\hat{v}(t, U(t))\right\|^{2p},\nonumber~~~~~~~~~~~~~~~~\\
		&\leq 2^{2p-1}\left\| 	\bar{u}^{n}(t)-	U(t)\right\|^{2p}+2^{2p-1}\left\|\hat{v}^n(t, \bar{u}^{n}(t))-\hat{v}(t, U(t))\right\|^{2p} ;
	\end{align*}
    Let $\bar{X}_q^{n}$ and 	$Y_q$ the restriction  of the solution $\bar{X}^{n}$ and $Y$ respectively on $D_q$ with $D=\cup_{q=1}^{\infty} D_q$,  then we have 
   \begin{align*}
		\left\| 	\bar{X}^{n}_q(t)-	Y_q(t)\right\|^{2p}&=\left\| 	\bar{u}_q^{n}(t)+\hat{v}^n(t, \bar{u}_q^{n}(t))-		U_q(t)-\hat{v}(t, U_q(t))\right\|^{2p},\nonumber~~~~~~~~~~~~~~~~\\
		&\leq 2^{2p-1}\left\| 	\bar{u}_q^{n}(t)-	U_q(t)\right\|^{2p}+2^{2p-1}\left\|\hat{v}^n(t, \bar{u}_q^{n}(t))-\hat{v}(t, U_q(t))\right\|^{2p}.
	\end{align*}
    From Lemma \ref{v} we have
     \begin{align*}
		\left\| 	\bar{X}^{n}_q(t)-	Y_q(t)\right\|^{2p}
        &\leq C\left(\frac{1}{n^{2p}}+\frac{\|\bar{u}^n_q(t)\|^{2p}}{n^{2p}} +\|\bar{u}^n_q(t)-\bar{u}^n_q(k_n(t))\|^{2p}\right)\\
        &+C\|U_q(t)-\bar{u}^n_q(t)\|^{2p}.
	\end{align*}
  Multiplying by $\exp{(-p\|\zeta\|) }$ and taking the expectation yields
   \begin{align*}
		\mathbb{	E}	\sup_{s\leq t}\left( \exp(-p\left\|\zeta \right\| )	\left\| 	\bar{X}_q^{n}(t)-	Y_q(t)\right\|^{2p}\right)
		&\leq \frac{C}{n^{2p}}+\mathbb{	E}	\sup_{s\leq t}\left( \exp(-p\left\|\zeta \right\| )	\frac{\|\bar{u}_q^{n}(t)\|^{2p}}{n^{2p}}\right)\\
        &+C\mathbb{	E}	\sup_{s\leq t}\left( \exp(-p\left\|\zeta \right\| )	\|\bar{u}_q^{n}(t)-\bar{u}_q^{n}(k_n(t))\|^{2p}\right)\\
        &+C\mathbb{	E}	\sup_{s\leq t}\left( \exp(-p\left\|\zeta \right\| )	\|U_q(t)-\bar{u}_q^{n}(t)\|^{2p}\right).
	\end{align*}
From Lemma \ref{lemef} and from \eqref{equa41a} in the Proof of Lemma \ref{theo3a} we have
 \begin{align}\label{bonqw}
		\mathbb{	E}	\sup_{s\leq t}\left( \exp(-p\left\|\zeta \right\| )	\left\| 	\bar{X}_q^{n}(t)-	Y_q(t)\right\|^{2p}\right)
		&\leq \frac{C}{n^{p}}
        +C\mathbb{	E}	\sup_{s\leq t}\left( \exp(-p\left\|\zeta \right\| )	\|U_q(t)-\bar{u}_q^{n}(t)\|^{2p}\right). 
	\end{align}
    Remember that in Lemma \ref{exis} we  have proved that $\bar{u}^n_q(t)$ converges  in probability to the continuous stochastic process  $U_q(t)$, uniformly in $t\in \left[0,T \right] $. 
    Since $(\bar{u}^n_q(t))_{n \geq 1}$ is a Cauchy sequence, from \eqref{equa17w} we have obtained that
    \begin{align*}
		\mathbb{	E}	\sup_{s\leq t}\left( \exp(-p\left\|\zeta \right\| )	\left\| \bar{u}_q^{n}(s)-	\bar{u}_q^{m}(s)\right\| ^{2p}\right) 
		&\leq C\left( \frac{1}{n^p}+\frac{1}{m^p}\right),t\in \left[0,T \right]. ~~~~~~~~~~~~~~~~~~~~~~~~~~~~~~~~~~
	\end{align*}
    This means that
	\begin{align*}
		\sup_m\mathbb{	E}	\sup_{s\leq t}\left( \exp(-p\left\|\zeta \right\| )	\left\| \bar{u}_q^{n}(s)-	\bar{u}_q^{m}(s)\right\| ^{2p}\right)
		&\leq C\left( \frac{1}{n^p}+1\right)<\infty. 
			\end{align*}
	Then for every $p\geq 2$, for $m\to \infty$ and $t\in \left[0,T \right]$  we have
	\begin{align}\label{equa18}
		\mathbb{	E}	\sup_{s\leq t}\left( \exp(-p\left\|\zeta \right\| )	\left\| \bar{u}^n_q(s)-	U_q(s)\right\| ^{2p}\right)
		&\leq \frac{C}{n^p}.
	\end{align}
    We substitute \eqref{equa18}  in \eqref{bonqw}	and we obtain
  \begin{align*}
		\mathbb{	E}	\sup_{s\leq t}\left( \exp(-p\left\|\zeta \right\| )	\left\| 	\bar{X}_q^{n}(t)-	Y_q(t)\right\|^{2p}\right)
		&\leq \frac{C}{n^{p}}. 
	\end{align*}
 Let  $\alpha>0$,  by using Markov's inequality, for $t\in \left[0,T \right] $, we have
	\begin{eqnarray}\label{equa19}
		&&\mathbb{	P}	\left( \sup_{s\leq t} \exp(-\frac{1}{2}\left\|\zeta \right\| )	\left\| \bar{X}^n_q(s)-	Y_q(s)\right\|\geq\frac{1}{n^{\alpha}}\right) \nonumber\\
		&\leq& 
        \cfrac{	\mathbb{	E}	\left( \sup_{s\leq t} \exp(-\frac{1}{2}\left\|\zeta \right\| )	\left\| \bar{X}^n_q(s)-	Y_q(s)\right\|\right) ^{2p}}{\left( \dfrac{1}{n^\alpha}\right) ^{2p}} \nonumber\\
        &\leq&	
		\mathbb{	E}	\left( \sup_{s\leq t} \exp(-p\left\|\zeta \right\| )	\left\| \bar{X}^n_q(s)-	Y_q(s)\right\|^{2p}\right)\dfrac{n^{2p\alpha} }{1} \nonumber\\
		&\leq& C\cfrac{n^{2p\alpha}}{n^p}= \frac{C}{n^{p-2p\alpha}} \nonumber.
	\end{eqnarray}
	For ${p-2p\alpha}>1$, we have $\alpha<\frac{1}{2}$. For $ \alpha \in \left(0,\frac{1}{2}\right)$, by  choosing $p>1$,  we therefore have
	\begin{eqnarray}
	    \label{equa19a}
		\sum_n	\mathbb{	P}	\left( \sup_{s\leq t} \exp(-\frac{1}{2}\left\|\zeta \right\| )	\left\| \bar{X}^n_q(s)-	Y_q(s)\right\|\geq\frac{1}{n^{\alpha}}\right)
		&\leq C\sum_n\frac{1}{n^{p-2p\alpha}}
		<\infty. \nonumber
	\end{eqnarray}
	 Note that  C is a constant that does not depend on $n$. 
      By following \cite[pp. 210]{gyongy1998note},  for any $p\geq 2$, we obtain
	\begin{equation*}
		\sup_{t\leq T}\left| \bar{X}_q^n(t)-Y_q(t)\right| \leq \beta n^{-\alpha}.
	\end{equation*}

Furthermore, following \cite[pp. 210]{gyongy1998note} we also have  for any $p\geq 2$ 
	\begin{equation}\label{equa21}
		\sup_{t\leq T}\left| \bar{X}^n(t)-Y(t)\right| \leq \beta n^{-\alpha}.
	\end{equation}
    This is done by using the fact that
    $$Y(t):=\lim_{q\to \infty} Y_q(t) $$ satisfies \eqref{equa7A} for $t\leq \tau$ where 
    $$\tau:=\lim_{q\to \infty } \tau^q=\inf\{t\geq 0: Y(t) \notin D\}\land T, \text{ with },$$
    $$\tau^q:=\inf\{t\geq 0: Y^q(t)\notin D_q\}\land T.$$
    The proof is therefore ended. 

\section{Simulation results}
The goal here is to give a realistic application in high dimensions and also  to check that our theoretical results are in agreement with experiments.
\subsection{Applications in high dimension}
	This example aims to show  that our numerical method can be applied even in high-dimension  problems. 
	Indeed  it is well known. 
	For  flow processes within the production phase, taking  in account the uncertainty in the input reaction function  and sub-grid fluctuations, it is well known (see \cite{SebaGatam}) that the unknown
 variable $X$ satisfies the following stochastic partial differential equation (SPDE) 
\begin{eqnarray}
\label{model}
d \phi Y(t) = [B(t)Y(t) + f(t, Y(t))]dt + g(t, Y(t))dW(t),\; \,Y(0) = Y_0\;\; t\in[0,T].
\end{eqnarray}
 The unknown variable $X$ represents the random  concentration of  the mixture of oil/gas and (CO2) (or any miscible product) during the enhanced recovery in the domain $\Omega$.
 In a geothermal reservoir, the unknown variable $X$  represents the random temperature distribution  of the water/ substance injected in the ground to enhance the recovery, 
 while $X$ represents the random concentration of the pollutant in groundwater reservoir. The function $g$ is the noise intensity  and the operator  $B(t)$ is an unbounded second order  operator,  generator of an analytic semi group. More precisely,  the operator $B(t)$ is the $L^2(\Omega)$ realisation  of  the differential operator $\mathcal{B}(t)$ given by
\begin{eqnarray}
\label{model1}
\mathcal{B}(t)Y=-\nabla \cdot (\mathbf{D}  \nabla  Y)+ \nabla \cdot (\mathbf{v} (t) Y).
\end{eqnarray}
 Note that  $\mathbf{D}$  is the  diffusion tensor and  $\mathbf{v} $ the  Darcy's velocity. The velocity field  $\mathbf{v} $ is usual found by solving a system of partial differential equations  from mass conservative  equation and Darcy 's law, see \cite{ATthesis} for more details. The SPDE \eqref{model} is solved using appropriate boundary, that can be Dirichlet, Neumann or mixed of Dirichlet and Neumann boundary conditions. The function  $\phi \in [0,1]$ is the porosity function measuring the fraction of void space.
 To allow the SPDE \eqref{model} to not be degenerated, it is usual  to take $\phi \in (0,1]$, here we consider the general  realistic case  $\phi\in [0,1]$. Note that $\phi =0$ means the presence of impermeable rocks such as stones, which cannot allow fluid to pass through. 
 After spatial approximation with finite element method or finite volume \cite{ATthesis},  with boundary  conditions,  we aims up  with  the following system  of  stochastic differential-algebraic equations in high dimension
 \begin{eqnarray}
\label{modelh}
 A_\mathcal{H} d Y_\mathcal{H}(t) = [B_\mathcal{H}(t)Y_\mathcal{H}(t) + f_\mathcal{H}(t, Y_\mathcal{H}(t))]dt + g_\mathcal{H}(t, Y_\mathcal{H}(t))dW_\mathcal{H}(t),~ Y(0) = Y_0.
\end{eqnarray}
where $A_\mathcal{H}$ is a diagonal matrix depending of $\phi$ usually called mass matrix \footnote{mass lumping technique \cite{masslumping} is using if the finite element method is using}.  If $\phi \in [0,1]$ is used, $A_\mathcal{H}$ is singular.

 In our simulation, we consider the problem on $\Omega=[0,2]\times [0,2] $, $T=1$, $\mathbf{D}=100 \mathbf{I}_{2,2}$\footnote{Identity $2\times 2$ matrix} and $v(t)=0$ with mixed of Dirichlet and Neumann boundary conditions, $\Gamma_{D_1}=\{0\}\times[0,2]$, $\Gamma_N\cup \Gamma _{D_1} =\partial\Omega$  and use the following coefficients,  boundary conditions and initial solution 
 
	\begin{eqnarray}\label{equaq}
	\left \{
	\begin{array}{l}
		Y(t,x,y)=1 ~~ \text{on} ~~ [0,  T]\times\Gamma_{D_1}~~~~~~~~~~~~~~~~~~~~~~~~~~~~~~~~~~~~~~\\
		\newline\\
		-\mathbf{D}\nabla Y(t,x,y) \cdot \mathbf{n}=0 ~~\text{on} ~~ [0,  T]\times\Gamma_N~~~~~~~~~~~~~~~~~~~~~~~~~~~~~~~~~~~~~~\\
		\newline\\
		Y(0,x,y)=0 ~in~ \Omega.~~~~~~~~~~~~~~~~~~~~~~~~~~~~~~~~~~~~~~ ~~~~~~~~~~~~\\
		\newline\\
		 f(t, Y(t))= Y-\phi\\
         \newline\\
		 g(t, Y(t))=10^{-4}\phi.
		\end{array}\right.
	\end{eqnarray}
Here we assume that $(A_\mathcal{H}-h B_\mathcal{H})$  is not singular and apply our scheme \eqref{equa2b} to the equation \eqref{modelh}. Indeed if $(A _\mathcal{H}-h B_\mathcal{H})$ is singular we can  always add and subtract a typical matrix $B$ such that 
$(A_\mathcal{H}-h B)$  is not singular. This is the idea behind our numerical technique which allows to solve efficiently large systems of  (stochastic) differential-algebraic equations without solving many algebraic equations at each time step.
%
		\begin{figure}
		\centering

        \begin{minipage}{0.55\textwidth}
		\centering
		(a)\includegraphics[width=\textwidth]{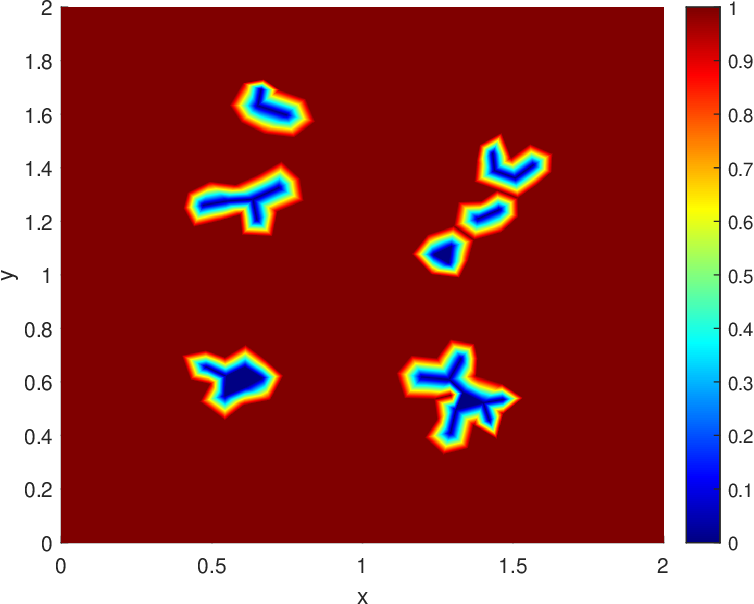}
	\end{minipage}

        	\centering
	\begin{minipage}{0.55\textwidth}
		\centering
		(b)\includegraphics[width=\textwidth]{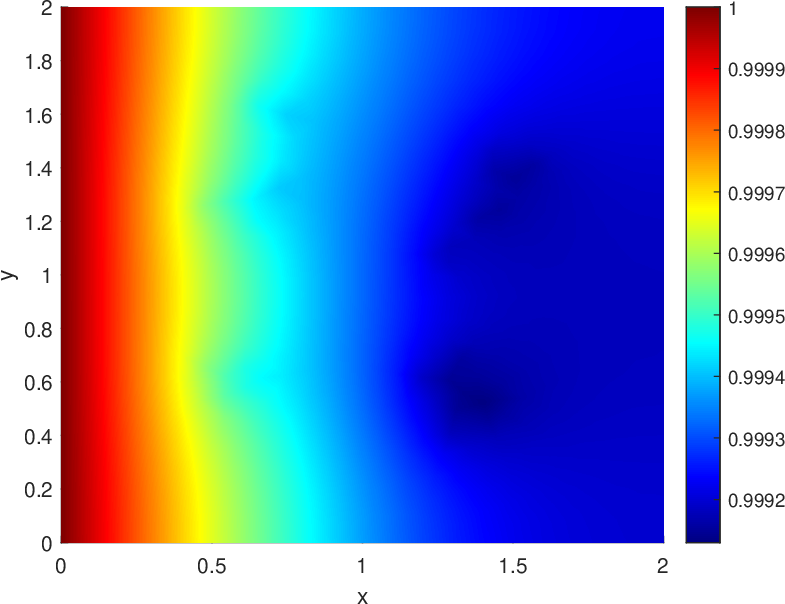}
	\end{minipage}
	\hfill
	\begin{minipage}{0.55\textwidth}
		\centering
		(c)\includegraphics[width=\textwidth]{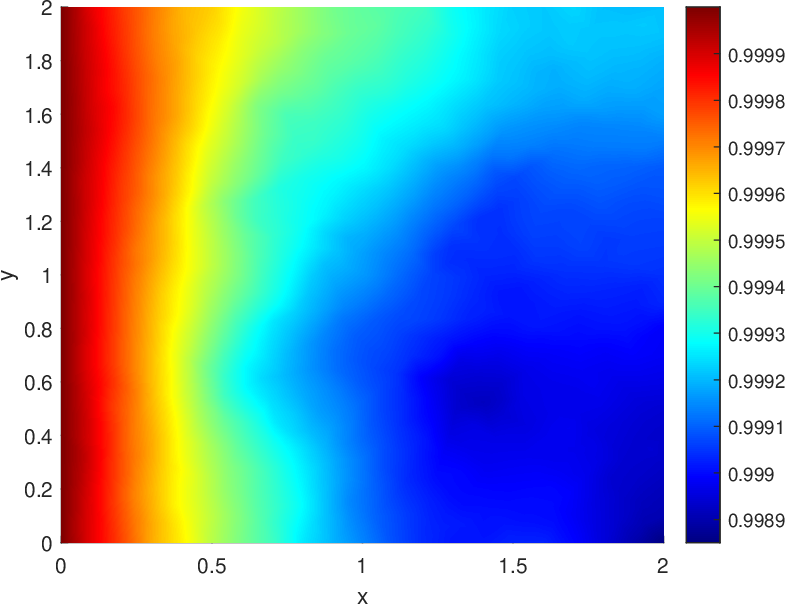}
	\end{minipage}
		\caption{Numerical simulation $ X^h $ using
			the finite element discretization for the mesh size $ 0.1 $ when the initial condition 
			$ X(0,x,y) = 0 $	and at time $ T = 1  $ (a) shows the Phi function, (b) shows the approximate solution without noise in 2D, (c) shows the approximate solution with the noise in 2D
            }
		\label{fig:grah1}
	\end{figure}
   \\
   
Figure \ref{fig:grah1} (a) we present the initial environment with impermeable rocks. Figure \ref{fig:grah1}(b) shows the propagation of the solution without noise.  We can  observe a deviation of the solution in the neighborhood of those rocks because of lower porosity.  Figure \ref{fig:grah1}  (c) illustrates the propagation  solution with  random noise with intensity $g$.  The difference between the two graphs comes from the fact that the noise acts as a source.
\subsection{Simulations in low dimension}
	The goal of this example is to check whether our theoretical convergence order for scheme \eqref{equa2b} given in Theorem \ref{theo3} is in agreement with computational order.
	 Here  $D=\mathbb{R}^3$ and consider the following index-1 SDAEs 
	\begin{equation}\label{solu1}
		\begin{pmatrix}
			1 & 0 &0  \\
			-1 & 0& t^2+1\\
			0&0&0
		\end{pmatrix}dY=f(t,Y(t))dt+g(t,Y(t))dW(t),~~ Y(0)=Y_0, ~~t\in\left[0,T \right],
	\end{equation}
	where 
	$$f(t,Y)=f(t,y_1,y_2,y_3) =\begin{pmatrix}
		-y_1-y_1^3  \\
		y_1+y_1^3+y_3\\
		y_1+y_2+y_3
	\end{pmatrix}$$ and $$ g(t,Y)=g(t,y_1,y_2,y_3)=\begin{pmatrix}
		\sqrt{2}y_1^2 & 0&0  \\
		-\sqrt{2}y_1^2+y_2 & y_1+y_3&y_1\\
		0&0&0
	\end{pmatrix}.$$
	Clearly, the matrix $A(t)=	\begin{pmatrix}
		1 & 0 &0  \\
		-1 & 0& t^2+1\\
		0&0&0
	\end{pmatrix}$ is  a singular matrix for  $t\in \left[0,T \right] $ and the functions $f$ and $g$ are local Lipschitz.
	
Let us show that the conditions from \cite[Theorem 1]{serea2025existence} are satisfied.  Indeed this will confirm that equation \eqref{solu1} has a unique solution.  We can check that the pseudo-inverse matrix $A^-$ should be
$$A^-(t)=	\begin{pmatrix}
	1 & 0 &0  \\
	0 & 0& 0\\
	\frac{1}{t^2+1}&\frac{1}{t^2+1}&0
\end{pmatrix}.$$

The projector matrix $P(t)$ is then given by $$P(t)=A^{-}A(t)=	\begin{pmatrix}
	1 & 0 &0  \\
	0 & 0& 0\\
	0&0&1
\end{pmatrix},$$

while  the projector matrix $R(t)$  is given by $$R(t)=I_{3\times 3}-AA^-(t)=	\begin{pmatrix}
	0 & 0 &0  \\
	0 & 0& 0\\
	0&0&1
\end{pmatrix}.$$

 Using the projection decomposition $Y(t)=U(t)+v(t)$, the constraint equation is given by
\begin{align*}
	A(t)v(t)+R(t)f(U(t)+v(t))&=	\begin{pmatrix}
		1 & 0 &0  \\
		-1 & 0& t^2+1\\
		0&0&0
	\end{pmatrix}	\begin{pmatrix}
	v_1(t)  \\
	v_2(t)\\
v_3(t)
	\end{pmatrix}+\begin{pmatrix}
	0 & 0 &0  \\
	0 & 0& 0\\
	0&0&1
	\end{pmatrix}\\
    &\times\begin{pmatrix}
	-(u_1+v_1)-(u_1+v_1)^3  \\
	(u_1+v_1)+(u_1+v_1)^3+(u_3+v_3)\\
	(u_1+v_1)+(u_2+v_2)+(u_3+v_3)
	\end{pmatrix}\\
	&=\begin{pmatrix}
		v_1(t)   \\
		-v_1(t)+(t^2+1)v_3(t)\\
		v_1(t)+u_1(t)+v_2(t)+u_2(t)+v_3(t)+u_3(t)
	\end{pmatrix}\\
	&=\begin{pmatrix}
		0   \\
		0\\
	0
	\end{pmatrix}.
\end{align*}
Then we obtain 
$$v_1(t)=0\;\;\; v_3(t)=0, \;\;\;\ v_2(t)=-u_1(t)-u_3(t)-u_2(t).$$
Consequentially, our constraint variable is not dependent on the noise and it is globally solvable. This means that equation \eqref{solu1} is an index-1 SDAEs.
 
More precisely, we can show that the Jacobian matrix $J$ of the constraint equation with respect to the variable $v$ is a nonsingular matrix and the norm of its inverse is bounded. Indeed  we have
$$J(t)=\begin{pmatrix}
	1 & 0 &0  \\
	-1 & 0& t^2+1\\
	1&1&1
\end{pmatrix}$$ with $$\det(J(t))=-(1+t^2)\ne 0, \text{ for all } t\in [0,T].$$

 Let us now   prove that the function $f(\cdot, \cdot)$ and $g(\cdot, \cdot)$ satisfy the monotone condition   with respect the variable $y$, that is 
$$\left\langle (P(t)Y)^T,A(t)^-f(t,Y)\right\rangle +\dfrac{1}{2}\left|A^-(t)g(t,Y) \right|^2_F \leq k(1+\left\|Y \right\|^2 ),~t\in \left[ 0, T \right], Y\in \mathbb{R}^3.$$
Indeed  we have
\begin{align*}
\left\langle (P(t)Y)^T,A(t)^-f(t,Y)\right\rangle&=\begin{pmatrix}
	y_1(t)  \\
	0\\
   y_3(t)
	\end{pmatrix} \begin{pmatrix}
	-y_1(t)-y_1^3(t)   \\
	0\\
	\frac{y_3}{t^2+1}
	\end{pmatrix}~~~~~~~~~~~~~~~~~~~~~~~~~~~~~~\\
	&=-y_1^2(t)-y_1^4(t)+\frac{y_3^2(t)}{t^2+1},
\end{align*}
with
\begin{align*}
	A^-(t)g(t,Y)&=\begin{pmatrix}
		\sqrt{2}y_1^2(t) & 0 &0  \\
		0&0&0\\
		\frac{y_2(t)}{1+t^2} & \frac{y_3(t)+y_1(t)}{1+t^2}& \frac{y_1(t)}{1+t^2}\\
	\end{pmatrix}\\
	|A^-(t)g(t)|^2&=2y_1^4+\frac{y_2^2}{(t^2+1)^2}+\frac{(y_3+y_1)^2}{(t^2+1)^2}+\frac{y_1^2}{(t^2+1)^2}~~~~~~~~~~~~~~~~~~~~~~~~~~~~~~~~~~~~\\
	&\leq 2y_1^4+\frac{y_2^2}{(t^2+1)^2}+\frac{2y_1^2+2y_3^2}{(t^2+1)^2}+y_1^2.
\end{align*}
Therefore for $Y\in \mathbb{R}^3$ and $t\in \left[ 0, T \right]$, we finaly have 
\begin{align*}
\left\langle (P(t)Y)^T,A(t)^-f(t,Y)\right\rangle +\dfrac{1}{2}\left|A^-(t)g(t,Y) \right|^2_F& \leq \frac{y_2^2}{2(t^2+1)^2}+\frac{y_1^2+y_3^2}{(t^2+1)^2}+\frac{y_3^2(t)}{1+t^2}\\
& \leq \frac{y_2^2}{2(t^2+1)}+\frac{y_1^2+y_3^2}{t^2+1}+\frac{y_3^2(t)}{1+t^2}\\
&\leq\frac{4}{2(t^2+1)} (1+\|Y\|^2).
\end{align*}

 Consequently, the solution $Y(\cdot)$  of equation \eqref{solu1}  exists and belongs to $\mathcal{M}^2(\left[0,T \right], \mathbb{R}^2 )$.
For our numerical simulation, we take $N=2^{18}$ , $T=1$ and $Y_0=(1,-2,1)$. 
	
	\begin{figure}[h]
	\centering
	\includegraphics[width=1\linewidth]{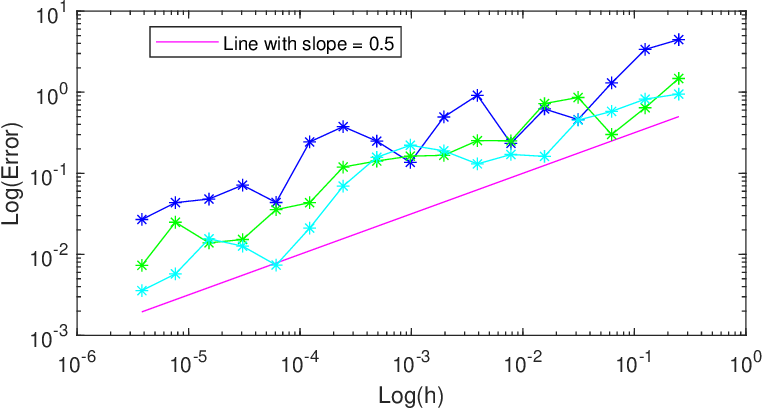}
	\caption{Pathwise convergence with our semi-implicit numerical method for three samples.}
	\label{fig:eroorr}
\end{figure}
	We  can observe from Figure \ref{fig:eroorr}, the error's curves  in log scale are not straight, this is due to the fact that the bounded constant $\beta$ in Theorem \ref{theo3} is a random variable. The convergence rates for these three different samples are 0.4010, 0.4344, and  0.4918. These three orders of convergence confirm our theoretical result  in Theorem \ref{theo3}.
	

	
\newpage
\textbf{Competing Interests:}
The authors declare that they have no competing interests.




\end{document}